\documentclass[a4paper, 10pt]{amsart}
\renewcommand{\subjclassname}{\textup{2000}~~\it Mathematics Subject Classification.~~}
\usepackage{amsfonts,amssymb}
\input{xy}
\xyoption{all}
\newcommand{\cctd}{\hfill$\square$\par\vspace{\baselineskip}}
\def\le{\langle}
\def\ri{\rangle}

\newtheorem{theorem}{Theorem}[section]
\newtheorem{lemma}[theorem]{Lemma}
\newtheorem{proposition}[theorem]{Proposition}
\newtheorem{corollary}[theorem]{Corollary}

\theoremstyle{definition}
\newtheorem{definition}[theorem]{Definition}
\newtheorem{example}[theorem]{Example}

\newtheorem{examples}[theorem]{Examples}

\theoremstyle{remark}

\newtheorem{remarks}[theorem]{Remarks}
\newcommand{\thlabel}[1]{\label{th:#1}}
\newcommand{\thref}[1]{Theorem~\ref{th:#1}}

\newcommand{\selabel}[1]{\label{se:#1}}

\newcommand{\seref}[1]{Section~\ref{se:#1}}

\newcommand{\prlabel}[1]{\label{pr:#1}}
\newcommand{\prref}[1]{Proposition~\ref{pr:#1}}
\newcommand{\colabel}[1]{\label{co:#1}}
\newcommand{\coref}[1]{Corollary~\ref{co:#1}}

\newcommand{\reslabel}[1]{\label{re:#1}}
\newcommand{\resref}[1]{Remarks~\ref{re:#1}}
\newcommand{\exlabel}[1]{\label{ex:#1}}
\newcommand{\exref}[1]{Example~\ref{ex:#1}}
\newcommand{\exslabel}[1]{\label{ex:#1}}
\newcommand{\exsref}[1]{Examples~\ref{ex:#1}}

\newcommand{\eqlabel}[1]{\label{eq:#1}}
\newcommand{\equref}[1]{(\ref{eq:#1})}

\def\equal#1{\smash{\mathop{=}\limits^{#1}}}

\def\<{\langle}
\def\>{\rangle}

\def\va{\varepsilon}
\def\v{\varphi}

\def\rh{\hbox{$\rightharpoonup$}}
\def\lh{\hbox{$\leftharpoonup$}}

\def\ra{\rightarrow}
\def\a{\alpha}
\def\b{\beta}

\def\l{\lambda}
\def\r{\rho}

\def\d{\delta}

\def\ov{\overline}
\def\un{\underline}

\def\le{\langle}
\def\ri{\rangle}

\newcommand{\tpla}{\mbox{$\tilde {p}^1$}}
\newcommand{\tplb}{\mbox{$\tilde {p}^2$}}
\newcommand{\tPla}{\mbox{$\tilde {P}^1$}}
\newcommand{\tPlb}{\mbox{$\tilde {P}^2$}}
\newcommand{\tqla}{\mbox{$\tilde {q}^1$}}
\newcommand{\tqlb}{\mbox{$\tilde {q}^2$}}
\newcommand{\tQla}{\mbox{$\tilde {Q}^1$}}
\newcommand{\tQlb}{\mbox{$\tilde {Q}^2$}}

\newcommand{\smi}{\mbox{$S^{-1}$}}

\def\rawo\lonra{\longrightarrow}

\def\ot{\otimes}

\def\cal{\mathcal}

\newcommand{\Cc}{\cal C}

\def\Id{{\rm Id}}
\def\equal#1{\smash{\mathop{=}\limits^{#1}}}

\def\equalupdown#1#2{\smash{\mathop{=}\limits^{#1}\limits_{#2}}}
\allowdisplaybreaks[4]

\begin{document}
\title[Integrals and cointegrals]
{On integrals and cointegrals for quasi-Hopf algebras}
\thanks{\subjclassname{\rm 16W30}.}
\author{Daniel Bulacu}
\address{Faculty of Mathematics and Informatics, University
of Bucharest, Str. Academiei 14, RO-010014 Bucharest 1, Romania}
\email{daniel.bulacu@fmi.unibuc.ro}
\author{Stefaan Caenepeel}
\address{Faculty of Engineering, 
Vrije Universiteit Brussel, B-1050 Brussels, Belgium}
\email{scaenepe@vub.ac.be}
\urladdr{http://homepages.vub.ac.be/\~{}scaenepe/}
\thanks{\rm 
The first author was supported by the strategic grant POSDRU/89/1.5/S/58852, Project 
``Postdoctoral programe for training scientific researchers" cofinanced by the European Social Found within the  Sectorial Operational Program Human Resources Development 2007 - 2013.
The second author
was supported by the research project G.0117.10  ``Equivariant
Brauer groups and Galois deformations'' from
FWO-Vlaanderen.}
\begin{abstract}
Using the machinery provided by a Frobenius algebra we show how the 
antipode of a quasi-Hopf algebra $H$ carries out left or right 
cointegrals for $H$. These formulas will allow us to find out the explicit form 
of a integral and a cointegral for the quantum double $D(H)$ of $H$ in terms of 
those of $H$, and so to answer to a conjecture of Hausser and Nill raised at the end 
of the nineties.    
\end{abstract}

\maketitle

%%%%%%%%%%%%%%%%%%%%%%%%%%%%%%%%%%%%%%%%%%%%%%%%%%%%%%%%%%%%
\section{Introduction}\selabel{intro}
%%%%%%%%%%%%%%%%%%%%%%%%%%%%%%%%%%%%%%%%%%%%%%%%%%%%%%%%%%%%
The main aim of this paper is to give an answer to the following conjecture 
raised by Hausser and Nill in \cite{hn3}: if $\l$ is a non-zero 
left cointegral on a finite dimensional quasi-Hopf algebra $H$ 
and $r$ is a non-zero right integral in $H$ then $\b\rh \l\Join r$ is a 
non-zero left integral in the quantum double $D(H)$ of $H$, at least if 
$H$ is unimodular; moreover, $D(H)$ is always a unimodular 
quasi-Hopf algebra. Here $\b$ and $\a$ are the distinguished elements of $H$ 
that appear in the definition of the antipode $S$ of $H$. 

It was shown in \cite[Theorem 6.5]{btfact} that the Drinfeld double $D(H)$ is
unimodular. This goes in two steps: first it is shown that $D(H)$ is a factorizable quasi-Hopf algebra, 
and then it is shown that every factorizable quasi-Hopf algebra is unimodular.
A categorical version of this result was given in  \cite{enoRadS4}, where it was shown
that a factorizable braided tensor category is unimodular. 
Anyway, in the quasi-Hopf algebra case this braided 
monoidal approach was necessary because an explicit description of an integral in 
$D(H)$ was not available at that time. We will provide such a desciption:
we will show 
that $\mu^{-1}(\d^2)\d^1\rh \l\Join r$ is a left and right integral 
in $D(H)$, where $\mu$ is the modular element of $H^*$ and 
$\d^1\ot \d^2$ is an element of $H\ot H$ that was defined by Drinfeld 
in \cite{dri}. When $H$ is unimodular we have that $\mu^{-1}(\d^2)\d^1=\b$, and
therefore 
the conjecture of Hausser and Nill is true in this case. Furthermore, 
using the Maschke theorem from \cite{pan} we obtain
that $D(H)$ is semisimple if and only if $H$ is semisimple and admits a 
normalized left cointegral, that is a left cointegral $\l$ on $H$ 
satisfying $\l(S^{-1}(\a)\b)\not=0$. This improves \cite[Corollary 8.3]{hn3}. 
Furthermore, we will describe explicitly the form of a left or right cointegral on $D(H)$ 
in terms of certain integrals and cointegrals for $H$, and then it will come out that 
$D(H)$ has a normalized left cointegral if and only if $H$ is semisimple 
and $H$ has a left normalized cointegral, if and only if $D(H)$ is a semisimple algebra. 
In particular, this gives a partial answer to a conjecture in \cite{btqd}.

A first question that we will deal with is the following: what is the appropriate notion
of a right cointegral on $H$? In classical Hopf algebra theory, this is a left cointegral
on $H^{\rm cop}$ or $H^{\rm op, cop}$, the opposite, respectively 
the opposite, co-opposite Hopf algebra associated to $H$. However, if $H$ is
a quasi-Hopf algebra, then the notions of left cointegral on $H^{\rm cop}$ and 
$H^{\rm op, cop}$ are different, and we have to decide which choice is the
appropriate one. We have chosen the ``cop''-version, motivated on one hand
by a result of Pareigis \cite{pareigis1} that was
recently actualized by Schauenburg \cite{schfreeness}, and on the other hand
by explicit formulas that we obtained in \prref{qhbimofduals}. This allows easy
characterizations for an element of $H^*$ to be a left or right cointegral on $H$.
Roughly speaking, these characterizations tell us that it suffices to check the
definining condition of a cointegral on a non-zero left integral in $H$, rather than
on all elements of $H$. As an application, we are able to compute 
the space of left cointegrals of $H(2)$ and $H_{\pm}(8)$,  
two quasi-Hopf algebras that have been introduced in \cite{eg3}.

${\cal L}$ and ${\cal R}$, the spaces of left and right cointegrals on $H$, are one-dimensional,
and therefore isomorphic. In \seref{S4forquasi}, we construct explicit isomorphisms,
using the antipode $\ov{S}$ of $H^*$ or its inverse. We use techniques coming from
the theory of Frobenius algebras, as developed in \cite{kad, sch}. The results can be
applied to compute the space of right cointegrals on the quasi-Hopf algebras
$H(2)$ and $H_{\pm}(8)$; moreover, we find some new formulas for the fourth power of the
antipode of $H$ in the case where ${\cal L}={\cal R}$. The images of left and right
integrals in $H$ under the antipode and its inverse will be presented in \prref{SinvSint}. 
Our final result is a description of the modular elements of $D(H)$ and $D(H)^*$ in
terms of the modular elements of $H$ and $H^*$.

We end our introduction with a philosophical note. Although the definition of 
quasi-Hopf algebras is - essentially - very natural, the explicit formulas and
computations are often quite technical. In order to streamline the storyline of
this paper, we therefore have decided to move some of the more technical
computations to an appendix, \seref{app}.

%%%%%%%%%%%%%%%%%%%%%%%%%%%%%%%%%%%%%%%%%%%%%%%
\section{Preliminary results}\selabel{prel}
%%%%%%%%%%%%%%%%%%%%%%%%%%%%%%%%%%%%%%%%%%%%%%%
We work over a commutative field $k$. All algebras, linear
spaces etc. will be over $k$; unadorned $\ot $ means $\ot_k$.
Following Drinfeld \cite{dri}, a quasi-bialgebra is
a four-tuple $(H, \Delta , \va , \Phi )$ where $H$ is
an associative algebra with unit,
$\Phi$ is an invertible element in $H\ot H\ot H$, and
$\Delta :\ H\ra H\ot H$ and $\va :\ H\ra k$ are algebra
homomorphisms satisfying the identities
\begin{eqnarray}
&&(\Id_H \ot \Delta )(\Delta (h))=
\Phi (\Delta \ot \Id_H)(\Delta (h))\Phi ^{-1},\eqlabel{q1}\\
&&(\Id_H \ot \va )(\Delta (h))=h~~,~~
(\va \ot \Id_H)(\Delta (h))=h,\eqlabel{q2}
\end{eqnarray}
for all $h\in H$, and
$\Phi$ has to be a $3$-cocycle, in the sense that
\begin{eqnarray}
&&(1\ot \Phi)(\Id_H\ot \Delta \ot \Id_H)
(\Phi)(\Phi \ot 1)\nonumber\\
&&\hspace*{1.5cm}
=(\Id_H\ot \Id_H \ot \Delta )(\Phi )
(\Delta \ot \Id_H \ot \Id_H)(\Phi),\eqlabel{q3}\\
&&(\Id \ot \va \ot \Id_H)(\Phi)=1\ot 1.\eqlabel{q4}
\end{eqnarray}
The map $\Delta$ is called the coproduct or the
comultiplication, $\va $ the counit and $\Phi $ the
reassociator. As for Hopf algebras we denote $\Delta (h)=h_1\ot h_2$,
but since $\Delta $ is only quasi-coassociative we adopt the
further convention (summation understood):
$$
(\Delta \ot \Id_H)(\Delta (h))=h_{(1, 1)}\ot h_{(1, 2)}\ot h_2~~,~~
(\Id_H\ot \Delta )(\Delta (h))=h_1\ot h_{(2, 1)}\ot h_{(2,2)},
$$
for all $h\in H$. We will
denote the tensor components of $\Phi$
by capital letters, and the ones of
$\Phi^{-1}$ by small letters,
namely
\begin{eqnarray*}
&&\Phi=X^1\ot X^2\ot X^3=T^1\ot T^2\ot T^3=
V^1\ot V^2\ot V^3=\cdots\\
&&\Phi^{-1}=x^1\ot x^2\ot x^3=t^1\ot t^2\ot t^3=
v^1\ot v^2\ot v^3=\cdots
\end{eqnarray*}
$H$ is called a quasi-Hopf
algebra if, moreover, there exists an
anti-morphism $S$ of the algebra
$H$ and elements $\a , \b \in
H$ such that, for all $h\in H$, we
have:
\begin{eqnarray}
&&
S(h_1)\a h_2=\va(h)\a
~~{\rm and}~~
h_1\b S(h_2)=\va (h)\b,\eqlabel{q5}\\ 
&&X^1\b S(X^2)\a X^3=1
~~{\rm and}~~
S(x^1)\a x^2\b S(x^3)=1.\eqlabel{q6}
\end{eqnarray}

Our definition of a quasi-Hopf algebra is different from the
one given by Drinfeld \cite{dri} in the sense that we do not
require the antipode to be bijective. In the case where $H$ is finite dimensional
or quasi-triangular, bijectivity of the antipode follows from the other axioms,
see \cite{bc1} and \cite{bn3}, so that both definitions coincide.

For later use, we now recall some examples of quasi-Hopf algebras; they
appeared for the first time in \cite{eg3}, and can be considered as the first
explicit examples of quasi-Hopf algebras. The integral and cointegral theory
that we will develop will be applied to these examples.

\begin{example}\exlabel{2dimAHA}
For $k$ a field of charactheristic different from $2$ 
let $H(2)$ be the two-dimensional Hopf group algebra 
generated by the grouplike element $g$ such that $g^2=1$. Since
$H(2)$ is commutative it can be also viewed as a quasi-Hopf algebra 
with reassociator $\Phi=1 -2p_{-}\ot p_{-}\ot p_{-}$, and  
antipode defined by $S(g)=g$ and distinguished 
elements $\a=g$ and $\b=1$. Here $p_{-}=\frac{1}{2}(1 - g)$.
\end{example}

\begin{example}\exlabel{8dimQHAs}
Consider $k$ a field that contains a primitive fourth root of unit $i$ (in 
particular, the characteristic of $k$ is not $2$). 
Let $H_{\pm}(8)$ be the unital algebra generated by $g, x$ with relations 
$g^2=1$, $x^4=0$ and $gx=-xg$, and endowed with the (non-coassociative) coalgebra 
structure given by the formulas 
$$\Delta(g)=g\ot g,~~\va(g)=1,~~
\Delta(x)=x\ot (p_{+}\pm ip_{-}) + 1\ot p_{+}x + g\ot p_{-}x,~~\va(x)=0,
$$
where $p_{\pm}=\frac{1}{2}(1\pm g)$. Then $H_{\pm}(8)$ are $8$-dimensional 
quasi-Hopf algebras with reassociator $\Phi=1 -2p_{-}\ot p_{-}\ot p_{-}$, 
antipode defined by $S(g)=g$ and $S(x)=-x(p_{+}\pm ip_{-})$, and distinguished 
elements $\a=g$ and $\b=1$. 
\end{example} 

More examples of quasi-Hopf algebras can be obtained by considering 
the opposite or coopposite construction. More precisely, 
if $H=(H, \Delta , \va , \Phi , S, \a , \b )$ is 
a quasi-bialgebra or a quasi-Hopf algebra then 
$H^{\rm op}$, $H^{\rm cop}$ and $H^{\rm op, cop}$ are also quasi-bialgebras 
(respectively quasi-Hopf algebras), where ``op" means opposite 
multiplication and ``cop" means opposite comultiplication. The 
structure maps are given by $\Phi _{\rm op}=\Phi ^{-1}$, 
$\Phi _{\rm cop}=(\Phi ^{-1})^{321}$, $\Phi _{\rm op, cop}=\Phi ^{321}$, 
$S_{\rm op}=S_{\rm cop}=(S_{\rm op, cop})^{-1}=S^{-1}$, $\a _{\rm op}=\smi (\b )$, 
$\b _{\rm op}=\smi (\a )$, $\a _{\rm cop}=\smi (\a )$, $\b _{\rm cop}=\smi (\b )$, 
$\a _{\rm op, cop}=\b $ and $\b _{\rm op, cop}=\a $.

The axioms of a quasi-Hopf algebra imply that 
$\va \circ S=\va $ and $\va (\a )\va (\b )=1$, 
so, by rescaling $\a $ and $\b $, we can assume without loss of generality
that $\va (\a )=\va (\b )=1$. The identities (\ref{eq:q2}-\ref{eq:q4})
also imply that
\begin{equation}\eqlabel{q7}
(\va \ot \Id_H\ot \Id_H)(\Phi )=(\Id_H \ot \Id_H\ot \va )(\Phi )=1\ot 1.
\end{equation}
It is well-known that the antipode of a Hopf algebra is an 
anti-coalgebra morphism. For a quasi-Hopf algebra, we have
the following statement: there exists an invertible element 
$f=f^1\ot f^2\in H\ot H$, called the Drinfeld twist or gauge transformation, such that 
$\va(f^1)f^2=\va(f^2)f^1=1$ and 
\begin{equation} \eqlabel{ca}
f\Delta (S(h))f^{-1}= (S\ot S)(\Delta ^{\rm cop}(h)),
\end{equation}
for all $h\in H$. $f$ can be described explicitly: 
first we define $\gamma, \delta\in H\ot H$ by
\begin{eqnarray}
&&\gamma=S(x^1X^2)\a x^2X^3_1\ot S(X^1)\a x^3X^3_2
\equal{(\ref{eq:q3},\ref{eq:q5})}
 S(X^2x^1_2)\a X^3x^2\ot S(X^1x^1_1)\a x^3,
\eqlabel{gamma}\\
&&\delta=X^1_1x^1\b S(X^3)\ot X^1_2x^2\b S(X^2x^3)
\equal{(\ref{eq:q3},\ref{eq:q5})} x^1\b S(x^3_2X^3)\ot x^2X^1\b S(x^3_1X^2).\eqlabel{delta}
\end{eqnarray}
With this notation $f$ and $f^{-1}$ are given by the formulas
\begin{eqnarray}
f&=&(S\ot S)(\Delta ^{\rm op}(x^1)) \gamma \Delta (x^2\b
S(x^3)),\eqlabel{f}\\ 
f^{-1}&=&\Delta (S(x^1)\a x^2) \delta
(S\ot S)(\Delta^{\rm cop}(x^3)).\eqlabel{g}
\end{eqnarray}
Moreover, $f$ satisfies the following relations:
\begin{eqnarray} 
&&f\Delta (\a )=\gamma~~,~~
\Delta (\b )f^{-1}=\delta ,\eqlabel{gdf}\\
&&f^1X^1\ot F^1f^2_!X^2\ot F^2f^2_2X^3= 
S(X^3)f^1F^1_1\ot S(X^2)f^2F^1_2\ot S(X^1)F^2,\eqlabel{pf}\\
&&g^1S(g^2\a )=\b 
~~,~~
S(\b f^1)f^2=\a
~~,~~
f^1\b S(f^2)=S(\a ),\eqlabel{fgab}
\end{eqnarray} 
where we have denoted $f=f^1\ot f^2=F^1\ot F^2$ and $f^{-1}=g^1\ot g^2$ as 
elements in $H\ot H$. The proof of these equations can be found in \cite{dri} 
and \cite[Lemma 2.6]{bn}.

We will need the appropriate generalization of the formula $h_1\ot h_2S(h_3)=h\ot 1$
in classical Hopf algebra theory. 

Following \cite{hn1, hn2}, we define
\begin{eqnarray}
&&p_R=p^1\ot p^2=x^1\ot x^2\b S(x^3),~~
q_R=q^1\ot q^2=X^1\ot S^{-1}(\a X^3)X^2,\eqlabel{qr}\\
&&p_L=\tpla \ot \tplb=X^2\smi (X^1\b )\ot X^3,~~
q_L=\tqla \ot \tqlb=S(x^1)\a x^2\ot x^3.\eqlabel{ql}
\end{eqnarray}
For all $h\in H$, we then have
\begin{eqnarray}
\Delta (h_1)p_R(1\ot S(h_2))&=&p_R(h\ot 1),\eqlabel{qr1}\\
(1\ot S^{-1}(h_2))q_R\Delta (h_1)&=&(h\ot 1)q_R,\eqlabel{qr1a}\\
\Delta (h_2)p_L(\smi (h_1)\ot 1)&=&p_L(1\ot h),\eqlabel{ql1}\\
(S(h_1)\ot 1)q_L\Delta (h_2)&=&(1\ot h)q_L.\eqlabel{ql1a}
\end{eqnarray}
Furthermore, the following relations hold
\begin{eqnarray}
&&(1\ot S^{-1}(p^2))q_R\Delta (p^1)=1\ot 1,\eqlabel{pqra}\\
&&\Delta (q^1)p_R(1\ot S(q^2))=1\ot 1,\eqlabel{pqr}\\
&&(S(\tpla)\ot 1)q_L\Delta (\tplb)=1\ot 1,\eqlabel{pql}\\
&&\Delta (\tqlb)p_L(\smi (\tqla)\ot 1)=1\ot 1.\eqlabel{pqla}
\end{eqnarray}
We also have that 
\begin{eqnarray}
&&X^1p^1_1P^1\ot X^2p^1_2P^2\ot X^3p^2\nonumber \\
&&\hspace{3cm}
=x^1_1p^1\ot x^1_{(2, 1)}p^2_1g^1S(x^3)\ot 
x^1_{(2, 2)}p^2_2g^2S(x^2),\eqlabel{pr1}\\
&&q^1Q^1_1x^1\ot q^2Q^1_2x^2\ot Q^2x^3\nonumber \\
&&\hspace{3cm}=
q^1X^1_1\ot \smi (f^2X^3)q^2_1X^1_{(2, 1)}\ot 
\smi (f^1X^2)q^2_2X^1_{(2, 2)},\eqlabel{qr2}\\
&&x^1\tilde{p}^1\ot x^2\tilde{p}^2_1\tilde{P}^1\ot x^3\tilde{p}^2_2\tilde{P}^2\nonumber \\
&&\hspace{3cm}= 
X^3_{(1, 1)}\tilde{p}^1_1\smi (X^2g^2)\ot X^3_{(1, 2)}\tilde{p}^1_2\smi (X^1g^1)
\ot X^3_2\tilde{p}^2,\eqlabel{pl1}\\
&&\tilde{Q}^1X^1\ot \tilde{q}^1\tilde{Q}^2_1X^2\ot \tilde{q}^2\tilde{Q}^2_2X^3\nonumber \\
&&\hspace{3cm}=
S(x^2)f^1\tilde{q}^1_1x^3_{(1, 1)}\ot S(x^1)f^2\tilde{q}^1_2x^3_{(1, 2)}\ot 
\tilde{q}^2x^3_2.\eqlabel{ql2}
\end{eqnarray}

%%%%%%%%%%%%%%%%%%%%%%%%%%%%%%%%%%%%%%%%%%%%%%%%%%%%%%%%%%%%%%%%
\section{Left and right cointegrals on a quasi-Hopf algebra}\selabel{cointegrals}
%%%%%%%%%%%%%%%%%%%%%%%%%%%%%%%%%%%%%%%%%%%%%%%%%%%%%%%%%%%%%%%
\setcounter{equation}{0}
Hausser and Nill \cite{hn3} have introduced the notion of left integral on
a finite dimensional quasi-Hopf algebra $H$. 
They define the space of left cointegrals ${\cal L}$ on $H$ as 
 the set of certain coinvariants associated to the quasi-Hopf 
$H$-bimodule $H^*$, the linear dual of $H$. Schauenburg \cite{schfreeness} 
has observed that the affiliation of $H^*$ 
to the category of right quasi-Hopf $H$-bimodules is a consequence 
of the following monoidal categorical result due Pareigis \cite{pareigis1}. 
For details on rigid monoidal categories, we refer the reader to \cite{k,maj}.

\begin{proposition}\prlabel{strgen}
Let $C$ be a coalgebra in the monoidal category $\Cc=(\Cc, \ot, a, \un{1}, l, r)$.\\
(i) Let  $(V,\lambda_V)$ be a left $C$-comodule
admitting a right dual ${}^*V$, with evaluation and coevaluation morphisms
${\rm ev}'_V$ and ${\rm coev}'_V$. Then ${}^*V$ is a right $C$-comodule,
with right $C$-coaction
\[
\hspace{3mm}\xymatrix{
{}^*V \ar[r]^-{l^{-1}_{{}^*V}}& \un{1}\ot {}^*V 
\hspace*{3mm}\ar[r]^-{{\rm coev}'_V\ot\Id_{{}^*V}}&\hspace*{3mm}
({}^*V\ot V)\ot {}^*V\hspace*{7mm}
\ar[r]^-{(\Id_{{}^*V}\ot \l_V)\ot\Id_{{}^*V}}&\hspace*{7mm}({}^*V\ot (C\ot V))\ot {}^*V
}
\]
\vspace{-3mm}
\[
\xymatrix{\hspace*{3mm} 
\ar[r]^-{a^{-1}_{{}^*V, V, C}\ot\Id_{{}^*V}}&\hspace{3mm}
(({}^*V\ot C)\ot V)\ot {}^*V
\hspace{3mm}\ar[r]^-{a_{{}^*V\ot C, V, {}^*V}}&
\hspace{3mm}({}^*V\ot C)\ot (V\ot {}^*V)
}
\]
\vspace{-3mm}
\[
\hspace*{-3.9cm}\xymatrix{
\ar[r]^-{\Id_{{}^*V\ot C}\ot {\rm ev}'_V}&\hspace{3mm}
({}^*V\ot C)\ot \un{1}\ar[r]^-{r_{{}^*V\ot C}}&{}^*V\ot C~~.
}
\]
In a similar way, if $(V,\rho_V)$ is a right $C$-comodule
admitting a left dual $V^*$, with evaluation and coevaluation morphisms
${\rm ev}_V$ and ${\rm coev}_V$, then $V^*$ is a left $C$-comodule,
with left $C$-coaction
\[
\hspace{3mm}\xymatrix{
V^*\ar[r]^-{r^{-1}_{V^*}}& V^*\ot \un{1}\hspace{3mm}\ar[r]^-{\Id_{V^*}\ot {\rm coev}_V}&
\hspace{3mm}V^*\ot (V\ot V^*)\hspace{7mm}\ar[r]^-{\Id_{V^*}\ot (\r_V\ot \Id_{V^*})}& 
\hspace{7mm}V^*\ot ((V\ot C)\ot V^*)
}
\]
\vspace{-3mm}
\[
\xymatrix{\hspace*{3mm} 
\ar[r]^-{\Id_{V^*}\ot a_{V, C, V^*}}&\hspace*{3mm} V^*\ot (V\ot (C\ot V^*))
\hspace*{5mm} \ar[r]^-{a^{-1}_{V^*, V, C\ot V^*}}& 
\hspace*{3mm} (V^*\ot V)\ot (C\ot V^*)
}
\]
\vspace{-3mm}
\[
\hspace{-3.9cm}\xymatrix{
\ar[r]^-{{\rm ev}_V\ot \Id_{C\ot V^*}}&\hspace*{3mm}\un{1}\ot (C\ot V^*)
\ar[r]^-{l_{C\ot V^*}}& C\ot V^*~.
}
\]
\end{proposition}

We apply this general result for the case when ${\cal C}={}_H{\cal M}_H$, 
the category of $H$-bimodules. This category is monoidal since it can be
identified with the category of left modules over the quasi-Hopf algebra 
$H^{\rm op}\ot H$. We provide the explicit construction of the monoidal
structure on ${}_H{\cal M}_H$.
\begin{itemize}
\item 
The associativity constraints 
$a_{M, N, P}: (M\ot N)\ot P\ra M\ot (N\ot P)$ are given by 
\begin{equation}\eqlabel{assbimqHA}
a_{M, N, P}((m\ot n)\ot p)=X^1\cdot m\cdot x^1\ot (X^2\cdot n\cdot x^2\ot 
X^3\cdot p\cdot x^3);
\end{equation}
\item the unit object is $k$ viewed as an $H$-bimodule via the counit $\va$ of $H$;
\item the left and right unit constraints are given by the natural isomorphisms 
$k\ot M\cong M\cong M\ot k$.
\end{itemize}
Recall that the linear dual $V^*$ of a right (resp. left) $H$-module $V$ is a left (resp. right) $H^*$-module
via $(h\rh v^*)(v)=v^*(v\cdot h)$ (resp. $(v^*\lh h)(v)=v^*(h\cdot v)$.\\
Let $\{v_i\}_i$ be a basis of a finite dimensional $H$-bimodule  $V$,
with dual basis $\{v^i\}_i$ of its linear dual $V^*$.
The left dual of  $V$ is  $V^*$,
regarded as an $H$-bimodule via
\begin{equation}\eqlabel{ldbimqHa1}
h\cdot v^*\cdot h'\equiv (h'\ot h)\cdot v^*=v^*\lh (S^{-1}(h')\ot S(h))
\equiv S^{-1}(h')\rh v^*\lh S(h),
\end{equation}
The evaluation morphism ${\rm ev}_V: V^*\ot V\ra k$ and the coevaluation morphism
${\rm coev}_V: k\ra V\ot V^*$ are given by the formulas
\begin{eqnarray}
&&{\rm ev}_V(v^*\ot v)\equiv v^*((S^{-1}(\b)\ot \a)\cdot v)
=v^*(\a \cdot v\cdot S^{-1}(\b));\eqlabel{ldbimqHa2}\\
&&{\rm coev}_V(1)\equiv \sum\limits_i (S^{-1}(\a)\ot \b)\cdot v_i\ot v^i
=\sum\limits_i \b\cdot v_i\cdot S^{-1}(\a)\ot v^i\eqlabel{ldbimqHa3}.
\end{eqnarray}
The right dual of $V$ is $V^*$, now with $H$-bimodule structure
\[
h\cdot v^*\cdot h'\equiv (h'\ot h)\cdot v^*=v^*\lh (S(h')\ot S^{-1}(h))
\equiv S(h')\rh v^*\lh S^{-1}(h)~~,
\]
The evaluation ${\rm ev}'_V: V\ot V^*\ra k$ and coevaluation ${\rm coev}'_V: k\ra V^*\ot V$
are now given by
\begin{eqnarray*}
&&{\rm ev}'_V(v^*\ot v)\equiv v^*((\b\ot S^{-1}(\a))\cdot v)\equiv 
v^*(S^{-1}(\a)\cdot v\cdot \b);\\
&& {\rm coev}'_V(1)\equiv\sum\limits_i v^i\ot (\a\ot S^{-1}(\b))\cdot v_i
\equiv \sum\limits_i v^i\ot S^{-1}(\b)\cdot v_i\cdot \a.
\end{eqnarray*}
Note that the two assertions follow easily from the canonical monoidal 
identification ${}_H{\cal M}_H\equiv {}_{H^{\rm op}\ot H}{\cal M}$ 
and the rigid monoidal structure of the category of finite dimensional 
representations over a quasi-Hopf algebra. The verification of all these 
details is left to the reader.

Via its (quasi) coalgebra structure $H$ has a natural (monoidal) coalgebra 
structure within ${}_H{\cal M}_H$. Recall that the category of right (left) 
quasi-Hopf $H$-bimodules is precisely the category of right (left) $H$-comodules 
within ${}_H{\cal M}_H$. The explicit definition of these concepts 
can be found in \cite{hn3}. Note that a left quasi-Hopf $H$-bimodule 
is nothing else than a right quasi-Hopf $H^{\rm cop}$-bimodule.

\begin{proposition}\prlabel{qhbimofduals}
Let $H$ be a finite dimensional quasi-Hopf algebra, with basis $\{e_i\}_i$,
and let $\{e^i\}_i$ be the corresponding dual basis of $H^*$.\\
(i) $H^*$ is a left quasi-Hopf $H$-bimodule via the structure 
$$\begin{array}{c}
h\cdot h^*\cdot h'=S^{-1}(h')\rh h^*\lh S(h);\\
\l_{H^*}(h^*)=\sum\limits_i h^*(S(\tplb)f^1(e_i)_1S^{-1}(\tqlb g^2))S(\tpla)f^2(e_i)_2
S^{-1}(\tqla g^1)\ot e^i.
\end{array}$$
Here $p_L=\tpla\ot\tplb$ and 
$q_L=\tqla\ot\tqlb$ are the elements defined in \equref{ql}, $f=f^1\ot f^2$ 
is the Drinfeld twist from \equref{f} and $f^{-1}=g^1\ot g^2$ is its inverse 
from \equref{g}.\\
(ii) $H^*$ is a right quasi-Hopf $H$-bimodule via the structure 
$$\begin{array}{c}
h\cdot h^*\cdot h'=S(h')\rh h^*\lh S^{-1}(h)~,\\
\r_{H^*}(h^*)=\sum\limits_i h^*(S^{-1}(f^1p^1)(e_i)_2g^2S(q^1))e^i\ot 
S^{-1}(f^2p^2)(e_i)_1g^1S(q^2),
\end{array}$$
where $p_R=p^1\ot p^2$ and 
$q_R=q^1\ot q^2$ are the elements presented in \equref{qr}. 
\end{proposition}

\begin{proof}
We will prove (i), and leave (ii) to the reader. Actually, we will show that
the structure on $H^*$ as stated in (i) is precisely the structure that we obtain
after applying part (ii) of \prref{strgen} in the case where $\Cc={}_H{\cal M}_H$ 
and $V=H$. Consider $H$ as a bimodule via multiplication.
As we have already mentioned the coassociativity of $\Delta$ expresses the fact 
that $H$ is a coalgebra within ${}_H{\cal M}_H$, and so a right $H$-comodule in $\Cc$. 
Since $H$ is finite dimensional we get that $H^*$, the left dual of $H$, is a left 
$H$-comodule in $\Cc$. By \equref{ldbimqHa1} we deduce that the 
$H$-bimodule structure of $H^*$ is the one mentioned in part (i) of the statement.  
In order to find the left $H$-coaction on $H^*$ we specialize \prref{strgen} (ii) 
for the monoidal structure of ${}_H{\cal M}_H$, and use \equref{ldbimqHa2} and
\equref{ldbimqHa3} to compute $\l_{H^*}$ as the following composition.
\begin{eqnarray*}
h^*
&\hspace*{-6mm}
\xymatrix{
\ar@{|->}[r]^-{\Id_{H^*}\ot {\rm coev}_H}& 
}
&
\hspace*{-3mm}
\sum\limits_ih^*\ot (\b e_iS^{-1}(\a)\ot e^i)
=
\sum\limits_{i, j}e^j(\b e_iS^{-1}(\a))h^*\ot (e_j\ot e^i)\\
&\hspace*{-6mm}
\xymatrix{
\ar@{|->}[r]^-{\Id_{H^*}\ot (\Delta\ot \Id_H)}&
}
&
\hspace*{-3mm}
\sum\limits_{i, j}e^j(\b e_iS^{-1}(\a))h^*\ot (((e_j)_1\ot (e_j)_2)\ot e^i)\\
&\hspace*{-6mm}
\xymatrix{
\ar@{|->}[r]^-{\Id_{H^*}\ot a_{H, H, H^*}}&
}
&
\hspace*{-3mm}
\sum\limits_{i, j}e^j(\b S(X^3)e_iS^{-1}(\a x^3))h^*\ot (X^1(e_j)_1x^1\ot 
(X^2(e_j)_2x^2\ot e^i))\\
&\hspace*{-6mm}
\xymatrix{
\ar@{|->}[r]^-{a^{-1}_{H^*, H, H\ot H^*}}&
}
&
\hspace*{-3mm}
\sum\limits_{i, j} e^j(\b S(y^3_2X^3)e_iS^{-1}(\a x^3Y^3_2))
\left(S^{-1}(Y^1)\rh h^*\lh S(y^1)\right.\\
&&\hspace*{2cm}\left.
\ot y^2X^1(e_j)_1x^1Y^2\right)\ot 
(y^3_1X^2(e_j)_2x^2Y^3_1\ot e^i)\\
&\hspace*{-6mm}
\xymatrix{
\ar@{|->}[r]^-{{\rm ev}_H\ot \Id_{H\ot H^*}}&
}
&
\hspace*{-3mm}
\sum\limits_{j}h^*(S(y^1)\a y^2X^1(e_j)_1x^1Y^2S^{-1}(Y^1\b))\\
&&\hspace*{2cm}
y^3_1X^2(e_j)_2x^2Y^3_1
\ot S^{-1}(\a x^3Y^3_2)\rh e^j\lh \b S(y^3_2X^3).
\end{eqnarray*}  
We then have, for all $h^*\in H^*$, that
\begin{eqnarray*}
\l_{H^*}(h^*)
&\equal{\equref{ql}}&\sum\limits_{j}\le h^*, \tqla X^1(e_j)_1x^1\tpla\ri \\
&&\hspace*{1cm}
\tilde{q}^2_1X^2(e_j)_2x^2\tilde{p}^2_1\ot S^{-1}(\a x^3\tilde{p}^2_2)\rh 
e^j\lh \b S(\tilde{q}^2_2X^3)\\
&=&\sum\limits_i \le h^*, \tqla X^1\b_1
S(\tilde{q}^2_2X^3)_1(e_i)_1S^{-1}(\a x^3\tilde{p}^2_2)_1x^1\tpla \ri \\
&&\hspace*{1cm}
\tilde{q}^2_1X^2\b_2S(\tilde{q}^2_2X^3)_2
(e_i)_2S^{-1}(\a x^3\tilde{p}^2_2)_2x^2\tilde{p}^2_1\ot e^i\\
&\equal{(\ref{eq:gdf},\ref{eq:ca})}&
\le h^*, \tqla X^1\d^1 S(\tilde{q}^2_{(2, 2)}X^3_2)f^1(e_i)_1S^{-1}(\gamma^2 x^3_2
\tilde{p}^2_{(2, 2)}g^2)x^1\tpla \ri\\
&&\hspace*{1cm}
\tilde{q}^2_1X^2\d^2 S(\tilde{q}^2_{(2, 1)}X^3_1)f^2(e_i)_2
S^{-1}(\gamma^1 x^3_1\tilde{p}^2_{(2, 1)}g^1)x^2\tilde{p}^2_1\ot e^i\\
&\equal{(\ref{eq:gamma},\ref{eq:delta})}&\sum\limits_i
\le h^*, \tqla \b S(\tilde{q}^2_{(2, 2)}X^3)f^1(e_i)_1
S^{-1}(\a x^3\tilde{p}^2_{(2, 2)}g^2)\tpla \ri \\
&&\hspace*{1cm}
\tilde{q}^2_1X^1\b S(\tilde{q}^2_{(2, 1)}X^2)f^2(e_i)_2S^{-1}(\a x^2\tilde{p}^2_{(2, 1)}g^1)
x^1\tilde{p}^2_1\ot e^i \\
&\equal{(\ref{eq:q1},\ref{eq:q5})}&\sum\limits_i 
\le h^*, \tqla\b S(X^3\tqlb)f^1(e_i)_1S^{-1}(\a \tplb x^3g^2)\tpla \ri \\
&&\hspace*{1cm} 
X^1\b S(X^2)f^2(e_i)_2S^{-1}(\a x^2g^1)x^1\ot e^i\\
&\equal{(\ref{eq:ql},\ref{eq:q6})}&\sum\limits_{i}
\le h^*, S(X^3)f^1(e_i)_1S^{-1}(x^3g^2)\ri \\
&&\hspace*{1cm}
X^1\b S(X^2)f^2(e_i)_2S^{-1}(\a x^2g^1)x^1\ot e^i\\
&\equal{\equref{ql}}&\sum\limits_i 
\le h^*, S(\tplb)f^1(e_i)_1S^{-1}(\tqlb g^2)\ri 
S(\tpla)f^2(e_i)_2S^{-1}(\tqla g^1)\ot e^i,
\end{eqnarray*}
as claimed. 
\end{proof}

If we denote
\begin{equation}\eqlabel{elemsUandV}
U=g^1S(q^2)\ot g^2S(q^1)~~{\rm and}~~
V=S^{-1}(f^2p^2)\ot S^{-1}(f^1p^1),
\end{equation}
then the right $H$-coaction on $H^*$ can be restated as 
\[
\r_{H^*}(h^*)=\sum\limits_i h^*(V^2(e_i)_2U^2)e^i\ot V^1(e_i)_1U^1
=\sum\limits_i e^i*h^*\ot e_i,
\]
where $U=U^1\ot U^2$, $V=V^1\ot V^2$ and $*$ is the - possibly non-associative - 
multiplication on $H^*$ given by 
$(\v * \psi)(h)=\v(V^1h_1U^1)\psi(V^2h_2U^2)$. This brings us to the
right quasi-Hopf $H$-bimodule structure on $H^*$ introduced in \cite{hn3}.
The coinvariants under this coaction are called left cointegrals on $H$. $\lambda\in H^*$
is coinvariant, i.e. a left cointegral if and only if
\begin{equation}\eqlabel{charactleftcoint}
\l(V^2h_2U^2)V^1h_1U^1=\mu(x^1)\l(hS(x^2))x^3,
\end{equation}
for all $h\in H$, or, according to \cite[Prop. 3.4 (c)]{bc1}
\begin{equation}\eqlabel{altcharactleftcoint}
\l(S^{-1}(f^1)h_2U^2)S^{-1}(f^2)h_1U^1=\mu(q^1_1x^1)\l(hS(q^1_2x^2))q^2x^3,
\end{equation}
for all $h\in H$. Here $\mu$ is the so-called modular element of $H^*$, which can be
introduced as follows. $t\in H$ is called a left (respectively right) 
integral in $H$ if $ht=\va (h)t$ (respectively $th=\va (h)t)$, for all $h\in H$.
$\int _l^H$ and $\int _r^H$, the spaces of left and right integral are one-dimensional
if $H$ is a finite dimensional quasi-Hopf algebra, see \cite{bc1, hn3}. They are also
ideals of $H$, so there exists $\mu\in H^*$ such that 
\begin{equation}\eqlabel{gdi} 
th=\mu(h)t,
\end{equation}
for all $t\in \int_l^H$ and $h\in H$. 
It can be easily checked that $\mu$ is an algebra map and that $\mu$ is convolution 
invertible with inverse $\mu^{-1}=\mu\circ S=\mu\circ S^{-1}$, and that 
\begin{equation}\eqlabel{gdim}
hr=\mu ^{-1}(h)r, 
\end{equation} 
for all $r\in \int _r^H$ and $h\in H$.
Furthermore, we have that
\begin{equation}\eqlabel{mumuinv}
\mu(\a\b)\mu^{-1}(\a\b)=\mu(X^1\b S(X^2)\a X^3)\mu^{-1}(S(x^1)\a x^2\b S(x^3))
\equal{\equref{q6}}1.
\end{equation}
If there exists a non-zero left integral in $H$ which is
at the same time a right integral, then $H$ is called
unimodular. Remark that $H$ is unimodular if and only if $\mu=\va$. 

\begin{examples}\exslabel{exsinteginQHAs}
1) Consider the $2$-dimensional quasi-Hopf algebra $H(2)$ constructed in 
\exref{2dimAHA}. It can be easily checked that $t=1+g$ is both a non-zero 
left and right integral in $H(2)$. Thus $H(2)$ is a unimodular 
quasi-Hopf algebra, and $\mu=\va$.\\
2) Let $H_{\pm}(8)$ be the two $8$-dimensional quasi-Hopf algebra 
considered in \exref{8dimQHAs}. If $t=(1+g)x^3$ then 
\begin{eqnarray*}
&&gt=g(1+g)x^3=(g+1)x^3=t=\va(g)t~~,~~{\rm and}\\
&&xt=x(1+g)x^3=(x+xg)x^3=(x-gx)x^3=(1-g)x^4=0=\va(x)t,
\end{eqnarray*}
and so $t$ is a non-zero left integral in $H_{\pm}(8)$, because 
$g$ and $x$ generate $H_{\pm}(8)$ as an algebra. In a similar way 
it can be shown that $r=(1-g)x^3$ is a non-zero right integral in 
$H_{\pm}(8)$. Since the characteristic of $k$ is different from $2$ it follows that 
$H_{\pm}(8)$ is not unimodular. It can be easily 
seen that the modular element of $H_{\pm}(8)^*$  is given by the relations 
$\mu(1)=1$, $\mu(g)=-1$ and $\mu(x)=0$.  
\end{examples}

The space of left cointegrals on a finite dimensional quasi-Hopf algebra $H$ 
is denoted by ${\cal L}$. It is proved in \cite{hn3} that ${\cal L}\ot H$ and $H^*$
are isomorphic as right quasi-Hopf $H$-bimodules; the functional corresponding to
$\lambda\ot h\in {\cal L}\ot H$ is $\l\cdot h=S(h)\rh \l$. From this isomorphism, it
follows easily that ${\cal L}$ is one-dimensional.

Now we introduce right cointegrals. We consider the quasi-Hopf algebra $H^{\rm cop}$,
and observe that $\gamma_{\rm cop}=(S^{-1}\ot S^{-1})(\gamma)$, and 
so $f_{\rm cop}=(S^{-1}\ot S^{-1})(f)$. It is easily seen that 
$(p_R)_{\rm cop}=\tplb\ot \tpla:=\tilde{p}_{21}$ and
$(q_R)_{\rm cop}=\tqlb\ot \tqla:=\tilde{q}_{21}$. 
Otherwise stated, the left quasi-Hopf $H$-bimodule structure of $H^*$ 
can be obtained from the right one by replacing $H$ by $H^{\rm cop}$. 
This is why we propose the following.

\begin{definition}
A right cointegral on a finite dimensional quasi-Hopf algebra $H$ is a 
left cointegral on $H^{\rm cop}$. 
\end{definition} 

The space of right cointegrals on $H$ will 
be denoted by ${\cal R}$. Since $H^{\rm cop}$ is also a 
finite dimensional quasi-Hopf algebra all the above results for left cointegrals
can be restated for right cointegral. For example $\Lambda\in H^*$ is a right
cointegral in $H^*$ if and only if
$$\Lambda(S(\tplb)f^1h_1S^{-1}(\tqlb g^2))S(\tpla)f^2h_2S^{-1}(\tqla g^1)=
\mu(X^3)\Lambda(hS^{-1}(X^2))X^1,$$
for all $h\in H$. We also have that $H\ot {\cal R}$ and $H^*$ are isomorphic as
left quasi-Hopf $H$-bimodules, where now $h\ot \Lambda$ corresponds to
$S^{-1}(h)\rh \Lambda$. Consequently 
any non-zero right cointegral is non-degenerate and ${\rm dim}_k{\cal R}=1$.\\

With an eye to examples, we now provide new equivalent characterizations of
left cointegrals. Perhaps, the easiest characterization of a left integral $\lambda$ on a Hopf
algebra $H$ is that $\l(t_2)t_1=\l(t)1$, where $t$ is a non-zero left integral in $H$.
Here one of the clue arguments is that $t$ generates $H$ as a left or right $H^*$-module.
If $H$ is a quasi-Hopf algebra, then $H^*$ is not associative, and we cannot
consider $H^*$-modules. However, we can still say that a left non-zero integral $t$ 
``generates" $H$, in the sense that the map 
\begin{equation}\eqlabel{Frobisomquasi}
\xi :\ H^*\ra H,~~
\xi(h^*)=(h^*\circ S)\ra t:=
h^*(S(q^2t_2p^2))q^1t_1p^1
\end{equation}
is bijective, cf. \cite{bc1}. Moreover, $\xi$ is a left $H$-module isomorphism between
$H$ and $H^*$.

From \cite{hn3}, we recall the relations
\begin{eqnarray}
&&U[1\ot S(h)]=\Delta(S(h_1))U[h_2\ot 1],~~\eqlabel{fu1}\\
&&[1\ot S^{-1}(h)]V=[h_2\ot 1]V\Delta(S^{-1}(h_1)),~~\eqlabel{fv1}\\
&&q_R=[\tqlb \ot 1]V\Delta(S^{-1}(\tqla)),~~{\rm and}\eqlabel{qqlv}\\
&&p_R=\Delta(S(\tpla))U[\tplb \ot 1].\eqlabel{pplu}
\end{eqnarray}
We will also need the following result.

\begin{lemma}
If $H$ is a finite dimensional quasi-Hopf algebra and $\mu$ is the modular element of $H^*$ 
then we have for all $\l\in {\cal L}$ and $t\in\int_l^H$ that
\begin{equation}\eqlabel{prelimpobs}
\l (q^2t_2p^2)q^1t_1p^1=\mu(\b)\l(t)1.
\end{equation}
\end{lemma} 

\begin{proof}
Let us start by noting that, for all $h, h'\in H$,
\begin{equation}\eqlabel{lcointsimpl}
\l (q^2h_2p^2S(h'))q^1h_1p^1=
\mu (x^1)
\l(S^{-1}(\tqla)hS(x^2h'_1\tpla))\tqlb x^3h'_2\tplb~~.
\end{equation}
Indeed, we have   
\begin{eqnarray*}
&&\hspace*{-2cm}\l (q^2h_2p^2S(h'))q^1h_1p^1\\
&\equal{(\ref{eq:qqlv},\ref{eq:pplu})}&
\l(V^2[\smi (\tqla )hS(\tpla)]_2U^2S(h'))\tqlb V^1[\smi (\tqla )hS(\tpla )]_1U^1\tplb \\
&\equal{\equref{fu1}}&\l(V^2[\smi (\tqla )hS(h'_1\tpla)]_2U^2)\tqlb V^1[\smi (\tqla )hS(h'_1\tpla )]_1U^1h'_2\tplb \\
&\equal{\equref{charactleftcoint}}&\mu (x^1)
\l(S^{-1}(\tqla)hS(x^2h'_1\tpla))\tqlb x^3h'_2\tplb.
\end{eqnarray*}
Specializing \equref{lcointsimpl} for $h=t$, a left integral in $H$, and $h'=1$ we obtain 
\begin{eqnarray*}
&&\hspace*{-2cm}
\l (q^2t_2p^2)q^1t_1p^1
\equal{\equref{ql}}\mu(x^1)\l(tS(x^2\tpla))x^3\tplb\\
&\equal{(\ref{eq:gdi},\ref{eq:ql})}&
\mu (x^1)\mu (X^1\b S(X^2))\mu
(S(x^2))\l (t)x^3X^3=\mu (\b )\l (t)1,
\end{eqnarray*}
and this finishes the proof. 
\end{proof}

\begin{theorem}\thlabel{charactleftcointvialeftint}
For a finite dimensional quasi-Hopf algebra $H$ and a non-zero 
element $\l\in H^*$ the following assertions are equivalent:
\begin{itemize}
\item[(i)] $\l$ is a left cointegral on $H$; 
\item[(ii)] $\l(q^2t_2p^2)q^1t_1p^1=\mu(\b)\l(t)1$, for any left integral 
$t$ in $H$;
\item[(iii)] $\l(t_2p^2)t_1p^1=\mu(\b)\l(t)\b$, for all $t\in \int_l^H$;
\item[(iv)] $\l(ht_2p^2)t_1p^1=\mu(\b)\l(t)\b S(h)$, for all $t\in \int_l^H$ and $h\in H$. 
\end{itemize}
Here $p_R$, $q_R$, $U$ and $V$ 
are the elements defined in \equref{qr}and \equref{elemsUandV}. 
\end{theorem}

\begin{proof}
$\un{(i)\Rightarrow (ii)}.$ follows from \equref{prelimpobs}.\\
$\un{(ii)\Rightarrow (i)}.$ Let $t$ be a non-zero left integral in $H$. 
If $\l(t)=0$ then $\l(Ht)=0$, and so ${\rm Ker}(\l)$ contains a non-zero 
ideal, contradicting the fact that any non-zero left cointegral on $H$ is non-degenerate. 
Thus $\l(t)\not=0$, and from here we get 
$(\mu(\b)\l(t))^{-1}\l(q^2t_2p^2)q^1t_1p^1=1$. Hence, by 
\cite[Lemma 6.2]{btfact} we conclude that ($\mu(\b)\l(t))^{-1}\l$, and therefore 
$\l$ itself, is a non-zero left cointegral on $H$.\\
$\un{(ii)\Rightarrow (iii)}.$ The formula 
\begin{equation}\eqlabel{f2}
t_1\ot S(t_2)=q^1t_1\ot  S(q^2t_2)\b =\b q^1t_1\ot S(q^2t_2).
\end{equation}
can be deduced easily from \equref{pqra} or can be found in \cite[Lemma 2.1]{bc1}. It 
implies that      
\[
\l(t_2p^2)t_1p^1=\l(q^2t_2p^2)\b q^1t_1p^1\equal{(ii)}\mu(\b)\l(t)\b .
\]
$\un{(iii)\Rightarrow (iv)}.$ By \equref{qr1} and \equref{gdi} it follows that 
\begin{equation}\eqlabel{movingelem1}
t_1p^1h\ot t_2p^2=\mu(h_1)t_1p^1\ot t_2p^2S(h_2),
\end{equation}
for all $h\in H$.
By \cite[Lemma 3.3]{bc1} we also have 
\begin{equation}\eqlabel{f4}
\l (S^{-1}(h)h')=\mu (h_1)\l (h'S(h_2)), 
\end{equation}
for all $\l\in{\cal L}$ and $h, h'\in H$.
We now have, for all $h\in H$, 
\begin{eqnarray*}
&&\hspace*{-2cm}
\l(ht_2p^2)t_1p^1=\le \l , S^{-1}(S(h))t_2p^2\ri t_1p^1
\equal{\equref{f4}}\mu(S(h)_1)\le \l, t_2p^2S(S(h)_2)\ri t_1p^1\\
&\equal{\equref{movingelem1}}&\l(t_2p^2)t_1p^1S(h)
\equal{(iii)}\mu(\b)\l(t)\b S(h).
\end{eqnarray*}
$\un{(iv)\Rightarrow (ii)}.$ We use $(iv)$ to compute that
\[
\l(q^2t_2p^2)q^1t_1p^1=\mu(\b)\l(t)q^1\b S(q^2)\equal{\equref{qr}}
\mu(\b)\l(t)X^1\b S(X^2)\a X^3\equal{\equref{q6}}\mu(\b)\l(t)1,
\]
completing the proof. 
\end{proof}

The characterizations in \thref{charactleftcointvialeftint} allow to find the left 
cointegrals on $H(2)$ and $H_{\pm}(8)$.
First we have 
to find a non-zero left integral in $H$. Secondly, working eventually 
with dual bases, we have to determine the element $\l\in H^*$ that 
satisfies, for instance, (iii), the most simple equivalent condition in 
\thref{charactleftcointvialeftint}. When $H$ is 
unimodular and $\a$ is invertible, then (iii) simplifies to
$\l(t_2)t_1=\l(t)\b\a$, because $\mu=\va$ and 
$t$, which this time is also a right integral satisfies 
\[
t_1\ot t_2
\equal{\equref{pqr}}
t_1q^1_1p^1\ot t_2q^1_2p^2S(q^2)=t_1p^1\ot t_2p^2S(\a)
\equal{\equref{movingelem1}}t_1p^1\a \ot t_2p^2.
\] 

\begin{example}\exlabel{lcointh2}
Let $\{P_1, P_g\}$ be the dual basis of $H(2)^*$ corresponding to
the basis $\{1, g\}$ of the quasi-Hopf algebra $H(2)$ from
\exref{2dimAHA}.  $P_g$ is a (non-zero) left cointegral on $H(2)$. 
\end{example}  
 
\begin{proof}
$H(2)$ is unimodular, $\b=1$, $\a=g$ is invertible, and $t=1 + g$ is a left
and right integral, see the first of \exsref{exsinteginQHAs}. So we have to
find the elements $\l\in H(2)^*$ satisfying 
$\l(1)1 + \l(g)g=\l(1 + g)g$. These clearly satisfy $\l(1)=0$, and so 
${\cal L}=kP_g$.   
\end{proof}

\begin{example}\exlabel{lcointh8}
Let $\{P_{g^ix^j}\mid 0\leq i\leq 1,~~0\leq j\leq 3\}$ be the dual basis 
corresponding to the canonical basis 
$\{g^ix^j\mid 0\leq i\leq 1,~~0\leq j\leq 3\}$ of $H_{\pm}(8)$. The space
of left cointegrals on the quasi-Hopf algebra
$H_{\pm}(8)$ from \exref{8dimQHAs} is $kP_{x^3}$.
\end{example}    

\begin{proof}
This time the computations are much more complicated. Recall first 
that $t=(1+g)x^3=x^3(1 - g)$ is a non-zero left integral in $H_{\pm}(8)$ 
which is not a right integral, see \exsref{exsinteginQHAs} 2).\\
Denote $\omega:=\frac{1}{2}(1\pm i)$ and let $\ov{\omega}=\frac{1}{2}(1\mp i)$ 
be its conjugate. In order to compute $\Delta(t)$ we rewrite $\Delta(x)$ as 
\[
\Delta(x)=\omega x\ot 1 + \ov{\omega} x\ot g + p_{+}\ot x + p_{-}\ot gx,
\]
to compute that 
\[
\Delta(x^2)=x^2\ot g + g\ot x^2 + (p_{+}\pm ip_{-})x\ot x + 
(p_{-}\pm ip_{+})x\ot gx,
\]
and then that 
\begin{eqnarray*}
&&\hspace*{-1.5cm}
\Delta(x^3)=\ov{\omega}x^3\ot 1 + \omega x^3\ot g \pm ip_{-}x^2\ot x 
+ \ov{\omega} gx\ot x^2 + p_{+}\ot x^3\\
&&\hspace*{1.5cm}
\pm ip_{+}x^2\ot gx - \omega gx\ot gx^2 
- p_{-}\ot gx^3.
\end{eqnarray*}
We have $\Delta(t)p_R=\Delta(x^3)\Delta(1-g)p_R$. Writing 
$\Phi=1 -2p_{-}\ot p_{-}\ot p_{-}$ under the form 
\begin{eqnarray*}
&&\hspace*{-1cm}
\Phi=\frac{3}{4}1\ot 1\ot 1 + \frac{1}{4}(1\ot 1\ot g + 1\ot g\ot 1 + g\ot 1\ot 1)\\
&&\hspace*{1.5cm}
- \frac{1}{4}(1\ot g\ot g + g\ot 1\ot g + g\ot g\ot 1) + \frac{1}{4}g\ot g\ot g,
\end{eqnarray*}
one can easily see that $\Phi^{-1}=\Phi$ and that 
\[
p_R=x^1\ot x^2\b S(x^3)=x^1\ot x^2x^3=\frac{1}{2}(1\ot 1 + 1\ot g + g\ot 1 - g\ot g).
\]
Now, using $\Delta(1-g)p_R=1\ot 1 - g\ot g$ we conclude that 
\begin{eqnarray*}
&&\hspace*{-1cm}
\Delta(t)p_R=(\ov{\omega}+\omega g)x^3\ot 1 + (\omega + \ov{\omega} g)x^3\ot g 
\pm i x^2\ot x + (\ov{\omega} g - \omega)x\ot x^2\\
&&\hspace*{2.5cm}
+ 1\ot x^3\pm i gx^2\ot gx - (\omega g - \ov{\omega})x\ot gx^2 + g\ot gx^3.
\end{eqnarray*}
Take now $\l=\sum\limits_{i,j}c_{ij}P_{g^ix^j}$ be an element of $H^*$. 
It follows that $\l$ satisfies (iii) in the statement of 
\thref{charactleftcointvialeftint} if and only if 
$c_{01}=c_{11}=c_{13}=0$ and the following relations hold,
\begin{eqnarray*}
&&\ov{\omega}c_{12} -\omega c_{02}=0~~,~~
\ov{\omega}c_{00} + \omega c_{10}=0~~,~~
\ov{\omega}c_{02} - \omega c_{12}=0~~{\rm and}~~
\omega c_{00} + \ov{\omega} c_{10}=0.
\end{eqnarray*} 
We find that $c_{00}=c_{02}=c_{10}=c_{12}=0$, and so $\l=c_{03}P_{x^3}$. 
We thus have that ${\cal L}=kP_{x^3}$, as stated. 
\end{proof}

Applying \thref{charactleftcointvialeftint} to $H^{\rm cop}$, we find the
following equivalent characterizations of right cointegrals.

\begin{corollary}\colabel{equivalentright}
Let $H$ be a finite dimensional quasi-Hopf algebra and $\Lambda$ a non-zero 
element of $H^*$. Then $\Lambda$ is a right cointegral on $H$ if and only if one of the 
equivalent relations below is satisfied:
$$\begin{array}{c}
\Lambda(\tqla t_1\tpla)\tqlb t_2\tplb =\mu^{-1}(\a)\Lambda(t)1,
~~{\rm for~all}~~t\in\int_l^H;\\
\Lambda(t_1\tpla)t_2\tplb=\mu^{-1}(\a)\Lambda(t)S^{-1}(\a),~~
{\rm for~all}~~t\in\int_l^H;\\
\Lambda(ht_1\tpla)t_2\tplb=\mu^{-1}(\a)\Lambda(t)S^{-1}(h\a),~~
{\rm for~all}~~t\in\int_l^H~~{\rm and}~~h\in H.
\end{array}$$
\end{corollary}

Examples of right cointegrals will be presented in the sequel. 

We end this section by recording that in any 
finite dimensional quasi-Hopf algebra $H$ we have 
\begin{equation}\eqlabel{qqt}
q^1t_1\ot q^2t_2=\tqla t_1\ot \tqlb t_2~~{\rm and}~~
r_1p^1\ot r_rp^2=r_1\tpla \ot r_2\tplb
\end{equation}
for all $t\in \int_l^H$ and $r\in \int_r^H$; the first formula appears in the  
proof of \cite[Lemma 6.1]{btfact} while the second one is the ``op" version of it. 
Thus in the equivalent conditions in \coref{equivalentright}, we can
interchange $q_R$ and $q_L$. 

%%%%%%%%%%%%%%%%%%%%%%%%%%%%%%%%%%%%%%%%%%%%%%%%%%%
\section{Integrals, cointegrals and the fourth power of the antipode}\selabel{S4forquasi}
%%%%%%%%%%%%%%%%%%%%%%%%%%%%%%%%%%%%%%%%%%%%%%%%%%%%%
\setcounter{equation}{0}
In this Section we present several Frobenius systems for a finite 
dimensional quasi-Hopf algebra $H$ in terms of integrals and cointegrals. 
Then we will see that the formula for the fourth power 
of the antipode $S$ proved in \cite{hn3,kad} simplifies in some  
particular situations. First we recall some equivalent conditions for a finite dimensional 
$k$-algebra $A$ to be Frobenius:
\begin{itemize}
\item $A$ and $A^*$ are isomorphic as right $A$-modules, where $A^*$ is a right 
$A$-module via $\le a^*\lh a, b\ri=a^*(ab)$, for all $a^*\in A^*$ and $a, b\in A$;
\item $A$ and $A^*$ are isomorphic as left $A$-modules, where $A^*$ is considered as 
a left $A$-module via the action $\le a\rh a^*, b\ri=a^*(ba)$, for all $a^*\in A^*$ 
and $a, b\in A$; 
\item there exists a pair $(\phi, e)$, called 
Frobenius pair or Frobenius system, with $\phi\in A^*$ and $e=e^1\ot e^2\in A\ot A$ 
(formal notation, summation implicitely understood), such that
\end{itemize}
\[
ae^1\ot e^2=e^1\ot e^2a,~~\forall~~a\in A,~~{\rm and}~~
\phi(e^1)e^2=\phi(e^2)e^1=1~.
\]  
The Frobenius system $(\phi, e)$ is unique in the 
following sense: any other Frobenius system for $A$ is either of the form 
$(\phi\lh d, e^1\ot d^{-1}e^2)$ or $({\bf d}\rh \phi, e^1{\bf d}^{-1}\ot e^2)$, for some suitable 
invertible elements $d, {\bf d}$ of $A$. 
Moreover, if $(\psi, f=f^1\ot f^2)$ is another Frobenius system for $A$ then 
\begin{equation}\eqlabel{dfromuniqanditsinv}
d=\psi(e^1)e^2~~,~~d^{-1}=\phi(f^1)f^2~~,~~{\bf d}=\chi^{-1}(d)~~{\rm and}~~{\bf d}^{-1}=\chi^{-1}(d^{-1}),
\end{equation}
where $\chi$ is the Nakayama automorphism associated to the Frobenius system $(\phi, e)$.
$\chi$ is the automorphism of $A$ uniquely determined by the equality $a\rh \phi=\phi\lh \chi(a)$, satisfied 
for all $a\in A$. It is well-known, see \cite{kad}, that 
\begin{equation}\eqlabel{NakAut}
\chi(a)=\phi(e^1a)e^2~~{\rm and}~~\chi^{-1}(a)=\phi(ae^2)e^1,
\end{equation}
for all $a\in A$.
Finally, let $\{a_i\}_i$ be a basis of $A$ with 
corresponding dual basis $\{a^i\}_i$ of $A^*$. If $f: A\ra A^*$ is
an isomorphism of right $A$-modules then 
$\big(f(1_A), \sum\limits_i a_i\ot f^{-1}(a^i)\big)$ is a Frobenius system 
for $A$. Likewise, if $f: A\ra A^*$ is a left $A$-linear isomorphism then 
$\big(f(1_A), \sum\limits_if^{-1}(a^i)\ot a_i\big)$ is a Frobenius system for $A$. 

We will now describe Frobenius systems for a finite dimensional quasi-Hopf algebra $H$.
Some of these systems also appear in \cite{kad}.

\begin{proposition}\prlabel{leftFrobQHA}
Let $H$ be a finite dimensional quasi-Hopf algebra and 
$(\l, t)\in {\cal L}\times \int_l^H$ satisfying $\l(S^{-1}(t))=1$. 
Then $(\l\circ S^{-1}, q^1t_1p^1\ot S(q^2t_2p^2))$ is a Frobenius 
system for $H$ with Nakayama automorphism given by $\chi(h)=\mu(h_1)S^2(h_2)$, for all $h\in H$.  
\end{proposition}

\begin{proof}
The properties of $\xi$ defined in \equref{Frobisomquasi} show that 
$H$ is a Frobenius algebra. On the other hand, as we mentioned in the 
proof of \thref{charactleftcointvialeftint}, the unique map $\l\in H^*$ 
satisfying $\l(q^2t_2p^2)q^1t_1p^1=1$ is a non-zero left cointegral on $H$. 
In terms of $\xi$ this means that $\xi(\l\circ S^{-1})=1$, and so 
$\l\circ S^{-1}$ is a Frobenius morphism for $H$ with Frobenius element 
\[
e=\sum\limits_i\xi(e^i)\ot e_i=q^1t_1p^1\ot S(q^2t_2p^2).
\]
Thus $\left(\l\circ S^{-1}, q^1t_1p^1\ot S(q^2t_2p^2)\right)$ is a Frobenius 
system for $H$, where $(\l , t)\in{\cal L}\times \int_l^H$ is such that 
$\l(q^2t_2p^2)q^1t_1p^1=1$ or, equivalently, $\l(S^{-1}(q^1t_1p^1))q^2t_2p^2=1$. 

Therefore, all we have to prove at this moment is that to give a pair 
$(\l, t)\in {\cal L}\times \int_l^H$ such that, for instance, 
$\l(S^{-1}(q^1t_1p^1))q^2t_2p^2=1$ is equivalent to giving a pair 
$(\l, t)\in {\cal L}\times \int_l^H$ such that $\l(S^{-1}(t))=1$.  
Indeed, if $\l(S^{-1}(q^1t_1p^1))q^2t_2p^2=1$ then applying $\va$ to the both sides 
we get $\l(S^{-1}(t))=1$. 

Conversely, since $S$ is an anti-algebra automorphism of $H$ by 
\cite[Lemma 3.1]{kad} we get that $(\l, q^2t_2p^2\ot S^{-1}(q^1t_1p^1))$ 
is another Frobenius system for $H$. From the uniqueness of the Frobenius system 
there exists an invertible element $\un{g}$ in $H$ such that 
\begin{equation}\eqlabel{firstRadformforquasi}
\l\circ S^{-1}=\l\lh \un{g}~~{\rm and}~~ 
q^1t_1p^1\ot S(q^2t_2p^2)=q^2t_2p^2\ot \un{g}^{-1}S^{-1}(q^1t_1p^1).
\end{equation}
Furthermore, by \equref{dfromuniqanditsinv} we have 
\begin{equation}\eqlabel{fu5}
\un{g}=\l(S^{-1}(q^2t_2p^2))S^{-1}(q^1t_1p^1)~~{\rm and}~~
\un{g}^{-1}=\l(q^1t_1p^1)S(q^2t_2p^2).
\end{equation}
Thus, if $\l\in {\cal L}$ such that $\l(S^{-1}(t))=1$ then 
\[
1=\l(S^{-1}(t))=(\l\lh \un{g})(t)=\l(\un{g}t)=\va(\un{g})\l(t)=\mu(\b)\l(S^{-1}(t))
\l(t)=\mu(\b)\l(t).
\]
Then \thref{charactleftcointvialeftint} says that $\l(q^2t_2p^2)q^1t_1p^1=1$ 
and as we already pointed out this is equivalent to 
$\l(S^{-1}(q^1t_1p^1))q^2t_2p^2=1$.

To compute $\chi$ we use the formula in \equref{NakAut} to obtain that
$$
\chi(h)=\phi(q^1t_1p^1h)S(q^2t_2p^2)
\equal{\equref{movingelem1}}\mu(h_1)\phi(q^1t_1p^1)S(q^2t_2p^2S(h_2))
=\mu(h_1)S^2(h_2),
$$
for all $h\in H$, or we can make use of \equref{f4}. This finishes the proof. 
\end{proof}

The element $\un{g}$ defined in \equref{fu5} is called the modular element of $H$. 
Together with $\mu$, it plays an important role in the structure of 
the fourth power of the antipode. An equivalent version of this formula can be 
easily obtained as follows. 

\begin{remarks}\reslabel{invxi}
(i) If $(\l, t)\in {\cal L}\times \int_l^H$ is such that $\l(S^{-1}(t))=1$, then the inverse of 
the map $\xi: H^*\ra H$ considered in \equref{Frobisomquasi} is given by 
$\xi^{-1}(h)=h\rh \l\circ S^{-1}$, for all $h\in H$. Indeed, $\xi$ is left $H$-linear 
so $\xi(h\rh \l\circ S^{-1})=h\xi(\l\circ S^{-1})=h\xi(\phi)=h$, for all $h\in H$.  
Thus the couple $(\l, t)$ has also the property that $t\rh \l\circ S^{-1}=\va$.  

Note also that $\l\circ S^{-1}=\l\lh \un{g}$ implies that
\[
\l\circ S=(\l\circ S^{-1}\lh \un{g}^{-1})\circ S=S^{-1}(\un{g}^{-1})\rh \l .
\]
Furthermore, since $S(\un{g})=\l\circ S^{-2}(S(q^2t_2p^2))q^1t_1p^1=\xi(\l\circ S^{-2})$ we deduce that 
\[
\l\circ S^{-2}=\xi^{-1}(S(\un{g}))=S(\un{g})\rh \l\circ S^{-1}=S(\un{g})\rh \l\lh \un{g}.
\]
(ii) For all $h\in H$ we have that
\begin{equation}\eqlabel{s4equivversion}
\mu(f^1)S^{-2}(h)S^{-1}(\un{g}^{-1})S(f^2)=\mu(h_1f^1)\mu^{-1}(h_{(2, 2)})S^{-1}(\un{g}^{-1})S(S(h_{(2, 1)})f^2),
\end{equation}
where, as usual, $f$ is the Drinfeld twist defined in \equref{f}. Observe that this equality  
can be viewed as an equivalent version of the formula for the fourth power 
of the antipode \cite{hn3,kad}, namely 
\[
S^4(\mu^{-1}\rh (h\lh \mu))=S^3(f^{-1}_\mu)S(\un{g})hS(\un{g}^{-1})S^3(f_\mu),
\]
for all $h\in H$. Here $f_{\mu}:=\mu(f^1)f^2$ and $h^*\rh h:=h^*(h_2)h_1$, for all 
$h^*\in H^*$ and $h\in H$.\\
To prove \equref{s4equivversion}, we compute $\l(S^{-1}(h)q^1t_1p^1)q^2t_2p^2$ in two different ways. 
On one hand, by \equref{qr1a} we have 
\begin{equation}\eqlabel{f1}
hq^1t_1\ot g^2t_2=q^1t_1\ot S^{-1}(h)q^2t_2,
\end{equation}
for all $h\in H$,
and therefore $\l(S^{-1}(h)q^1t_1p^1)q^2t_2p^2=\l(q^1t_1p^1)S^{-2}(h)q^2t_2p^2=S^{-2}(h)S^{-1}(\un{g})$, 
for all $h\in H$. On the other hand, 
\begin{eqnarray*}
&&\hspace*{-20mm}
\l(S^{-1}(h)q^1t_1p^1)q^2t_2p^2\equal{\equref{f4}}
\mu(h_1)\l(q^1t_1p^1S(h_2))q^2t_2p^2\\
&\equal{\equref{movingelem1}}&
\mu(h_1)\mu(S(h_2)_1)\l(q^1t_1p^1)q^2t_2p^2S(S(h_2)_2)\\
&=&\mu(h_1)\mu(S(h_2)_1)S^{-1}(\un{g}^{-1})S(S(h_2)_2).
\end{eqnarray*}
Comparing the two equalities above by \equref{ca} we obtain \equref{s4equivversion}.  
\end{remarks}

Another Frobenius system for $H$ can be obtained by working with $H^{\rm cop}$ 
instead of $H$. This will allow us to find a bijection between the spaces of left and 
right cointegrals on $H$.

\begin{proposition}\prlabel{deforuforint}
Let $H$ be a finite dimensional quasi-Hopf algebra, 
$t$ a non-zero left integral in $H$ and let $\l\in {\cal L}$ and 
$\Lambda\in {\cal R}$ be such that $\l(S^{-1}(t))=1$ and 
$\Lambda(S(t))=1$, respectively. Then $u:=\mu(V^1)S^2(V^2)$ is invertible 
in $H$ and $\l\circ S^{-1}=\Lambda\lh u$. Consequently the map
${\cal L}\to {\cal R}$ sending $\l$ to $\l\circ S^{-1}\lh u^{-1}$ is a
well-defined bijection.
\end{proposition} 

\begin{proof}
Applying \prref{leftFrobQHA} with $H$ replaced by $H^{\rm cop}$, we find that
$(\Lambda\circ S, \tqlb t_2\tplb\ot S^{-1}(\tqla t_1\tpla))$,
with $\Lambda$ the unique right cointegral on $H$ such that $\Lambda(S(t))=1$,
is a Frobenius system for $H$. 
By \cite[Lemma 3.1]{kad} we have that 
$(\Lambda, \tqla t_1\tpla\ot S(\tqla t_2\tpla))$ is also a Frobenius system for 
$H^{\rm cop}$, and therefore for $H$ as well. 
Now, we know that the Frobenius systems 
$(\l\circ S^{-1}, q^1t_1p^1\ot S(q^2t_2p^2))$ and 
$(\Lambda, \tqla t_1\tpla\ot S(\tqlb t_2\tplb))$ are related through 
an invertible element $u\in H$ satisfying 
\begin{equation}\eqlabel{qtrversustqlattpla}
\l\circ S^{-1}=\Lambda\lh u~~,~~
q^1t_1p^1\ot S(q^2t_2p^2)=\tqla t_1\tpla\ot u^{-1}S(\tqlb t_2\tplb). 
\end{equation}
In order to prove the first assertion, it therefore suffices to show that 
$u=\mu(V^1)S^2(V^2)$. 
To this end, we first apply \equref{dfromuniqanditsinv} and obtain
that $u=\l(S^{-1}(\tqla t_1\tpla))S(\tqlb t_2\tplb)$.
Using \equref{fgab} we compute  
\begin{eqnarray*}
f^1\tpla\ot f^2\tplb&\equal{\equref{ql}}&f^1X^2S^{-1}(X^1\b)\ot f^2X^3\\
&\equal{({\ref{eq:fgab},\ref{eq:q5}})}&
f^1X^2g^1_2G^2\a S^{-1}(X^1g^1_1G^1)\ot f^2X^3g^2\\
&\equal{\equref{pf}}&
f^1g^2_1G^1S(X^2)\a X^3S^{-1}(g^1)\ot f^2g^2_2G^2S(X^1)\\
&\equal{({\ref{eq:ca},\ref{eq:qr}})}&
S(q^2S^{-1}(g^2)_2)S^{-1}(g^1)\ot S(q^1S^{-1}(g^2)_1),
\end{eqnarray*}
where $f=f^1\ot f^2=F^1\ot F^2$ and $f^{-1}=g^1\ot g^2=G^1\ot G^2$. From
\equref{qr1} it follows that any right integral $r$ in $H$ satisfies  
\begin{equation}\eqlabel{elemmovedbyrightint}
r_1p^1h\ot r_2p^2=r_1p^1\ot r_2p^2S(h),
\end{equation}
for all $h\in H$. 
In particular, if we set $r:=S^{-1}(t)\in \int_r^H$ then  
\begin{eqnarray*}
t_1\tpla \ot t_2\tplb&\equal{\equref{ca}}&g^1S(r_2)f^1\tpla \ot g^2S(r_1)f^2\tplb \\
&=&g^1S(q^2S^{-1}(G^2)_2r_2)S^{-1}(G^1)\ot g^2S(q^1S^{-1}(G^2)_1r_1)\\
&\equal{({\ref{eq:gdim},\ref{eq:g}})}&\mu(G^2)\a_1\d^1S(q^2r_2)S^{-1}(G^1)\ot \a_2\d^2S(q^1r_1)\\
&\equal{\equref{delta}}&\mu(G^2)\a_1\b S(q^2r_2X^3)S^{-1}(G^1)\ot \a_2X^1\b S(q^1r_1X^2)\\
&\equal{({\ref{eq:ql},\ref{eq:qqt}})}&
\mu(G^2)\a_1\b S(q^2r_2p^2)S^{-1}(G^1)\ot \a_2 S(q^1r_1p^1)\\
&\equal{\equref{elemmovedbyrightint}}&\mu(G^2)\a_1\b S(q^2r_2p^2\a_2)S^{-1}(G^1)
\ot S(q^1r_1p^1)\\
&\equal{\equref{q5}}&\mu(G^2)\b S(q^2r_2p^2)S^{-1}(G^1)\ot S(q^1r_1p^1),
\end{eqnarray*}
and therefore
\begin{eqnarray*}
\tqla t_1\tpla \ot \tqlb t_2\tplb 
&=&\mu(G^2)\tqla\b S(q^2r_2p^2)S^{-1}(G^1)\ot \tqlb S(q^1r_1p^1)\\
&\equal{\equref{elemmovedbyrightint}}&\mu(G^2)\tqla\b S(q^2r_2p^2\tqlb)S^{-1}(G^1)
\ot S(q^1r_1p^1)\\
&\equal{({\ref{eq:ql},\ref{eq:q6}})}&\mu(G^2)S(q^2r_2p^2)S^{-1}(G^1)\ot S(q^1r_1p^1).
\end{eqnarray*}
Then we compute that
\begin{eqnarray*}
u&=&\mu(G^2)\l(S^{-2}(G^1)q^2r_2p^2)S^2(q^1r_1p^1)\\
&\equal{\equref{f4}}&\mu(g^2)\mu(S^{-1}(g^1)_1)\l(q^2r_2p^2S(S^{-1}(g^1)_2))S^2(q^1r_1p^1)\\
&\equal{\equref{elemmovedbyrightint}}&
\mu(g^2)\mu(S^{-1}(g^1)_1)\l(q^2r_2p^2)S^2(q^1r_1p^1S^{-1}(g^1)_2).
\end{eqnarray*}
Using the fact that $r$ is right integral in $H$, we find that
$$\l(q^2r_2p^2)q^1r_1p^1
\equal{\equref{lcointsimpl}}\mu(x^1)\l(S^{-1}(\tqla)rS(x^2\tpla))\tqlb x^3\tplb
\equal{\equref{gdim}}\l(S^{-1}(t))\tqlb=\mu(\tqla)\tqlb,$$
and then we can finally compute that
\begin{eqnarray*}
&&\hspace*{-2cm}
u=\mu(g^2\tqla)\mu(S^{-1}(g^1)_1)S^2(\tqlb S^{-1}(g^1)_2)\\
&\equal{({\ref{eq:ql},\ref{eq:ca}})}&
\mu(g^2S(x^1)\a S^{-1}(f^2g^1_2G^2S(x^2)))
S(f^1g^1_1G^1S(x^3))\\
&\equal{({\ref{eq:pf},\ref{eq:fgab}})}&
\mu(x^3S^{-1}(f^2x^2\b))S(f^1x^1)\equal{\equref{qr}}\mu^{-1}(f^2p^2)S(f^1p^1)
\equal{\equref{elemsUandV}}\mu(V^1)S^2(V^2),
\end{eqnarray*}
as needed. Using \equref{dfromuniqanditsinv} or a simple observation, we see that
\begin{equation}\eqlabel{invofaspecelem}
u^{-1}=\mu^{-1}(q^1_2g^2S(q^2))S(q^1_1g^1).
\end{equation}
From the uniqueness of left and right cointegrals on $H$, it follows that the map
${\cal L}\to {\cal R}$ in the statement is well-defined. Its inverse sends
$\Lambda\in {\cal R}$ to $(\Lambda \lh u)\circ S\in {\cal L}$.
\end{proof}

\begin{corollary}
If $H$ is a finite dimensional unimodular quasi-Hopf algebra and 
$\ov{S}$ is the antipode of $H^*$ then 
$\ov{S}^{-1}({\cal L})={\cal R}$.   
\end{corollary}

\begin{proof}
$\mu=\va$ implies $u=1$, and so $\l\circ S^{-1}=\Lambda$. Everything then follows 
from the uniqueness of left and right cointegrals on $H$. 
\end{proof}

\begin{corollary}
Consider $(\Lambda, t)\in {\cal R}\times \int_l^H$ such that $\Lambda(S(t))=1$. Then 
$\Lambda\circ S=\Lambda\lh uS^2(S^{-1}(u^{-1})\lh \mu)\un{g}^{-1}$, where 
$\un{g}$ is the modular element of $H$, and where we define $h\lh h^*:=h^*(h_1)h_2$, for 
all $h^*\in H^*$ and $h\in H$. 
\end{corollary}

\begin{proof}
If $(\l, t)\in{\cal L}\times \int_l^H$ satisfies $\l(S^{-1}(t))=1$ then we 
know that $\l\circ S^{-1}=\l\lh \un{g}$. Now we write this property in $H^{\rm cop}$:
for any couple $(\Lambda, t)\in {\cal R}\times \int_l^H$ such that $\Lambda(S(t))=1$, 
we have that $\Lambda\circ S=\Lambda \lh \un{g}_{\rm cop}$ with 
$\un{g}_{\rm cop}=\Lambda(S(\tqla t_1\tpla))S(\tqlb t_2\tplb)$. The proof is
finished after we show that 
$\un{g}_{\rm cop}=uS^2(S^{-1}(u^{-1})\lh \mu)\un{g}^{-1}$. 

By $\l\circ S^{-1}=\Lambda\lh u$ we get 
$\Lambda \circ S=(\l\circ S^{-1}\lh u^{-1})\circ S=S^{-1}(u^{-1})\rh \l$, so 
\begin{eqnarray*}
\un{g}_{\rm cop}&=&\l(\tqla t_1\tpla S^{-1}(u^{-1}))S(\tqlb t_2\tplb)\\
&\equal{\equref{qtrversustqlattpla}}&\l(q^1t_1p^1S^{-1}(u^{-1}))uS(q^2t_2p^2)\\
&\equal{\equref{movingelem1}}&\mu(S^{-1}(u^{-1})_1)\l(q^1t_1p^1)uS^2(S^{-1}(u^{-1})_2)S(q^2t_2p^2)\\
&\equal{\equref{fu5}}&uS^2(S^{-1}(u^{-1})\lh \mu)\un{g}^{-1},
\end{eqnarray*} 
as desired. Notice that $\un{g}_{\rm cop}=\un{g}^{-1}$ in the case when $u=1$, and 
this happens for instance when $H$ is unimodular.
\end{proof}

\begin{corollary}\colabel{modularforunimod}
Let $H$ be a finite dimensional quasi-Hopf algebra and assume that ${\cal L}={\cal R}$. 
Then $\un{g}=\mu(\b)\mu^{-1}(\b)^{-1}u$. Consequently, if $H$ is unimodular 
and admits a non-zero left cointegral that is at the same time right cointegral 
then $\un{g}=1$.
\end{corollary}

\begin{proof}
Since ${\rm dim}_k{\cal L}={\rm dim}_k{\cal R}=1$ it follows that 
${\cal L}={\cal R}$ if and only if ${\cal L}\cap {\cal R}\not=0$. 

Let $0\not=\l\in {\cal L}\cap {\cal R}$ and $t\in \int_l^H$ such that 
$\l(S^{-1}(t))=1$. Then $\l\circ S^{-1}=\Lambda\lh u$, for some non-zero 
$\Lambda\in {\cal R}$. But $\Lambda =c\l$ for a certain $c\in k$, so 
$\l\circ S^{-1}=c\l\lh u$. We have $\mu(\b)\l(t)=1$, hence 
\[
1=\l(S^{-1}(t))=c\l(ut)=c\va(u)\l(t)=c\mu^{-1}(\b)\mu(\b)^{-1}.
\]
We get $c=\mu(\b)\mu^{-1}(\b)^{-1}$ and therefore, using \equref{f1}, we conclude that 
\begin{eqnarray*}
&&\hspace*{-1cm}
\un{g}=\l(S^{-1}(q^2t_2p^2))S^{-1}(q^1t_1p^1)=
c\l(uq^2t_2p^2)S^{-1}(q^1t_1p^1)\\
&&\hspace*{1cm}
\equal{\equref{f1}}c\l(q^2t_2p^2)S^{-1}(q^1t_1p^1)u=
\mu(\b)\mu^{-1}(\b)^{-1}u,
\end{eqnarray*} 
as stated. 
\end{proof}

The above formulas tell us how to find right cointegrals from left cointegrals, and vice-versa.

\begin{example}
For $H(2)$ we have that $P_g$ is at the same time left and right cointegral 
and $\un{g}=1$.
\end{example}

\begin{proof}
$H(2)$ is unimodular and has the antipode defined by the identity map. 
Thus in this particular case the formula $\l\circ S^{-1}=\Lambda\lh u$ 
reduces to $\l=\Lambda$, and so ${\cal L}={\cal R}$. From \exref{lcointh2} 
we deduce that $P_g$ is a left and right non-zero cointegral on $H(2)$, and 
from \coref{modularforunimod} we get $\un{g}=1$. 
\end{proof}

\begin{example}
For $H_{\pm}(8)$ we have ${\cal R}=k(\omega P_{x^3} + \ov{\omega}P_{gx^3})$, 
$\un{g}=\omega 1 + \ov{\omega} g$ and $\un{g}^{-1}=\ov{\omega} 1 + \omega g$, 
respectively.  
\end{example}

\begin{proof}
To find a right cointegral on $H_{\pm}(8)$ we compute 
$\l\circ S^{-1}$ and the element $u$. Then 
$\l\circ S^{-1}\lh u^{-1}$ will be a non-zero right cointegral on $H_{\pm}(8)$. 

Consider the left integral $t=(1+g)x^3$ and take $\l=cP_{x^3}$ 
with $c\in k$ that has to be determined such that $\l(S^{-1}(t))=1$. Actually, 
since $\b=1$ we need to find that unique $c\in k$ such that $\l(t)=1$ and 
it then follows that we should have $c=1$, and thus $\l=P_{x^3}$.  

We use now $(p_{+}\pm ip_{-})(p_{+}\mp ip_{-})=1$ to see that 
$S^{-1}(x)=-(p_{+}\mp ip_{-})x$, and 
\begin{eqnarray*}
&&S^{-1}(x^2)=\mp ix^2~~,~~S^{-1}(x^3)=\pm i(p_{+}\mp ip_{-})x^3~~,~~
S^{-1}(gx)=(p_{+}\pm ip_{-})x~~,~~\\
&&S^{-1}(gx^2)=\mp igx^2~~,~~S^{-1}(gx^3)=\mp i (p_{+}\pm i p_{-})x^3~~. 
\end{eqnarray*}
In particular, we get $\l(S^{-1}(g^ix^j))=0$ unless in the following two cases 
when we have 
\begin{eqnarray*}
&&\l(S^{-1}(x^3))=\pm i\l(p_{+}x^3) + \l(p_{-}x^3)=\frac{1}{2}(1\pm i)=\omega~~,
{\rm and~~respectively}~~\\
&&\l(S^{-1}(gx^3))=\mp i\l(p_{+}x^3) + \l(p_{-}x^3)=\frac{1}{2}(\mp i + 1)=\ov{\omega}~~.
\end{eqnarray*}
In other words, we have $\l\circ S^{-1}=\omega P_{x^3} + \ov{\omega}P_{gx^3}$. 
It can be easily checked that $f=f^{-1}=p_R$ in the case when $H=H_{\pm}(8)$. 
We conclude that $u=1$, even if $H$ is not 
unimodular. Thus $\omega P_{x^3} + \ov{\omega}P_{gx^3}$ is a right non-zero cointegral 
on $H_{\pm}(8)$.  

Let us finally compute $\un{g}$. 
Since $\b=1$ formula \equref{f2} implies 
$q^2t_2p^2\ot q^1t_1p^1=t_1p^1\ot t_2p^2$, hence 
$\un{g}=\l(S^{-1}(t_2p^2))S^{-1}(t_1p^1)$. 
By the expression of $\Delta(t)p_R$ found in \exref{lcointh8} we then obtain 
\[
\un{g}=\l(S^{-1}(x^3))1 + \l(S^{-1}(gx^3))g=\omega 1 + \ov{\omega} g, 
\]
as desired. A simple inspection shows that $\un{g}^{-1}=\ov{\omega}1 + \omega g$, 
and this completes the proof. 
\end{proof}

We now investigate the relation between $\l\circ S$ and $\Lambda$. 

\begin{proposition}
Let $t$ be a left integral in $H$ and let $\l\in {\cal L}$ and 
$\Lambda\in {\cal R}$ be such that $\l(S^{-1}(t))=1$ and 
$\Lambda(S(t))=1$. Then 
$v=(\mu^{-1}(\un{g})\mu(\b))^{-1}\mu(S(p^2)f^1)S(p^1)f^2$
is invertible in $H$ and $\l\circ S=\Lambda\lh v$. 
Consequently, we have a well-defined bijection between ${\cal L}$ and
${\cal R}$ mapping $\l$ to $\l\circ S\lh v^{-1}$.
\end{proposition}

\begin{proof}
Applying \prref{leftFrobQHA} to $H^{\rm op}$, we find that 
$(\phi_{\rm op}, q^1r_1p^1\ot S^{-1}(q^2r_2p^2))$ is a Frobenius system for $H$; 
$r$ is a non-zero right integral in $H$, and $\phi_{\rm op}$ 
is the unique element of $H^*$ satisfying 
$\phi_{\rm op}(S^{-1}(q^2r_2p^2))q^1r^1p^1=1$ or, equivalently, 
$\phi_{\rm op}(q^1r_1p^1)q^2r_2p^2=1$. Now set $t=S(r)\in \int_l^H$ 
and take $\l\in {\cal L}$ and $\Lambda\in {\cal R}$ such that 
$\l(S^{-1}(t))=1$ and $\Lambda(S(t))=1$. We will prove that 
$\phi_{\rm op}=\mu^{-1}(\tpla)S^{-2}(\tplb)\rh \l\circ S$. 
To this end recall from the proof of \cite[Lemma 6.1]{btfact} that  
\begin{eqnarray}
&&V^1r_1U^1\ot V^2r_2U^2=S^{-1}(q^2t_2p^2)\ot S^{-1}(q^1t_1p^1),\eqlabel{tsFrobelem}\\
&&U=\tilde{q}^1_1p^1\ot \tilde{q}^1_2p^2S(\tqlb)~~{\rm and}~~
V=q^1\tilde{p}^1_1\ot S^{-1}(\tplb)q^2\tilde{p}^1_2.\eqlabel{UVpql}
\end{eqnarray}
From \equref{UVpql} it follows that
\begin{eqnarray*}
V^1r_1U^1\ot V^2r_2U^2&=&q^1\tilde{p}^1_1r_1\tilde{q}^1_1p^1\ot 
S^{-1}(\tplb)q^2\tilde{p}^1_2r_2\tilde{q}^1_2p^2S(\tqlb)\\
&=&\mu^{-1}(\tpla)q^1r_1p^1\ot S^{-1}(\tplb)q^2r_2p^2,
\end{eqnarray*}
and then \equref{tsFrobelem} can be restated as 
\[
\mu^{-1}(\tpla)q^1r_1p^1\ot S^{-1}(\tplb)q^2r_2p^2=
S^{-1}(q^2t_2p^2)\ot S^{-1}(q^1t_1p^1).
\]
Observe now that, for all $h\in H$, 
\[
hq^1r_1\ot q^2r_2\equal{\equref{qr1a}}q^1h_{(1, 1)}r_1\ot S^{-1}(h_2)q^2h_{(1, 2)}r_2
\equal{\equref{gdim}}\mu^{-1}(h_1)q^1r_1\ot S^{-1}(h_2)q^2r_2.
\]
Since  
\[
\mu(\tqla)\mu^{-1}(\tpla)\tqlb q^1r_1p^1\ot S^{-1}(\tplb)q^2r_2p^2=
\mu(\tqla)S^{-1}(q^2t_2p^2S(\tqlb))\ot S^{-1}(q^1t_1p^1),
\]
this implies that
\begin{equation}\eqlabel{qrpversusqtp}
q^1r_1p^1\ot q^2r_2p^2=\mu(\tqla)\tqlb S^{-1}(q^2t_2p^2)\ot S^{-1}(q^1t_1p^1),
\end{equation}
where we made use of \equref{pqla}. 

In the proof of \prref{deforuforint} we have shown that 
$\l(q^2r_2p^2)q^1r_1p^1=\mu(\tqla)\tqlb$, hence 
\begin{eqnarray*}
&&\hspace*{-2cm}
1\equal{\equref{pql}}\mu(S(\tpla)\tqla \tilde{p}^2_1)\tqlb \tilde{p}^2_2
=
\mu^{-1}(\tpla)\mu(\tilde{p}^2_1)\l(q^2r_2p^2)q^1r_1p^1\tilde{p}^2_2\\
&\equal{\equref{elemmovedbyrightint}}&
\mu^{-1}(\tpla)\mu(\tilde{p}^2_1)\l(q^2r_2p^2S(\tilde{p}^2_2))q^1r_1p^1\\
&\equal{\equref{f4}}&
\mu^{-1}(\tpla)\l(S^{-1}(\tplb)q^2r_2p^2)q^1r_1p^1\\
&=&\le \l\lh \mu^{-1}(\tpla)S^{-1}(\tplb), q^2r_2p^2\ri q^1r_1p^1.
\end{eqnarray*}
From the uniqueness of the map $\phi_{\rm op}$ it follows that 
\[
\phi_{\rm op}=(\l\lh \mu^{-1}(\tpla)S^{-1}(\tplb))\circ S=
\mu^{-1}(\tpla)S^{-2}(\tplb)\rh \l\circ S,
\]
as we claimed. By the definition of $p_L$ it is immediate that 
$d:=\mu^{-1}(\tpla)S^{-2}(\tplb)$ is invertible in $H$, and therefore 
$(\l\circ S, q^1r_1p^1d\ot S^{-1}(q^2r_2p^2))$ is a Frobenius 
system for $H$, whenever $(\l, r)\in {\cal L}\times \int_r^H$ is 
such that $\l(r)=1$. Comparing it with 
$(\Lambda, \tqla t_1\tpla \ot S(\tqlb t_2\tplb))$ 
we conclude that there is an invertible element $v\in H$ such that 
$\l\circ S=\Lambda\lh v$. In order to compute $v$, observe first that 
\equref{ql}, \equref{pf} and the formula 
$\smi (f^2)\b f^1=\smi (\a )$ imply
\begin{equation}\eqlabel{formtplfversusqg}
S(\tplb )f^1\ot S(\tpla )f^2=q^1g^1_1\ot \smi (g^2)q^2g^1_2,
\end{equation}
where, as usual, we denote $f^{-1}=g^1\ot g^2$.

According to \equref{dfromuniqanditsinv} we have that
\begin{eqnarray*}
v&=&\l(S(\tqla t_1\tpla))S(\tqlb t_2\tplb)\\
&\equal{\equref{ca}}&
\le \l, S(\tpla)f^2S(t)_2g^2S(q^1)\ri S(\tplb)f^1S(t)_1g^1S(q^2)\\
&\equal{(\ref{eq:formtplfversusqg},\ref{eq:elemsUandV})}&
\le \l , S^{-1}(g^2)q^2(g^1S(t))_2U^2\ri q^1(g^1S(t))_1U^1\\
&\equal{(\ref{eq:f4},\ref{eq:gdim})}&
\mu(g^2_1)\mu^{-1}(g^1)\le \l, q^2S(t)_2U^2S(g^2_2)\ri q^1S(t)_1U^1\\
&\equal{(\ref{eq:fu1},*)}&
\mu(g^2_1)\mu^{-1}(g^1)
\le \l, q^2S(t)_2U^2\ri q^1S(t)_1U^1g^2_2\\
&\equal{(\ref{eq:qqlv},\ref{eq:gdim})}&  
\mu^{-1}(g^1)\mu(\tqla g^2_1)\le \l, V^2S(t)_2U^2\ri \tqlb V^1S(t)_1U^1g^2_2\\
&\equal{(\ref{eq:charactleftcoint},*)}&
\mu(S(g^1)\tqla g^2_1)\l(S(t))\tqlb g^2_2.
\end{eqnarray*}
At (*) we used that $S(t)\in \int_r^H$. Now we claim that
\begin{equation}\eqlabel{fpformula}
S(g^1)\tqla g^2_1\ot \tqlb g^2_2=S(p^2)f^1\ot S(p^1)f^2.
\end{equation}
Indeed, using \equref{ql} and \equref{pf}, we compute 
\begin{eqnarray*}
&&\hspace*{-2cm}
S(g^1)\tqla g^2_1\ot \tqlb g^2_2=
S(g^1_1G^1S(x^3))\a g^1_2G^2S(x^2)f^1\ot g^1S(x^1)f^2\\
&\equal{\equref{q5}, \equref{fgab}}&
S(x^2\b S(x^3))f^1\ot S(x^1)f^2=S(p^2)f^1\ot S(p^1)f^2.
\end{eqnarray*}
We also have that $\l(S(t))=(S^{-1}(\un{g}^{-1})\rh \l)(t)=
\mu^{-1}(\un{g}^{-1})\l(t)=(\mu^{-1}(\un{g})\mu(\b))^{-1}$, hence
$v=(\mu^{-1}(\un{g})\mu(\b))^{-1}\mu(S(p^2)f^1)S(p^1)f^2$. 
It is easy to see that the inverse of $v$ is
$v^{-1}=\mu^{-1}(\un{g})\mu(\b q^2g^1S(q^1_2))g^2S(q^1_1)$, 
and this finishes the proof. 
\end{proof}

\begin{corollary}\colabel{antipcointunimod}
Let $\ov{S}$ be the antipode of the dual of a finite dimensional unimodular quasi-Hopf algebra
$H$. Then $\ov{S}({\cal L})={\cal R}$. 
\end{corollary}

\begin{proof}
In this case we have $\mu=\va$ and therefore $v=1$. 
\end{proof} 

For $a\in H$ invertible, let  ${\rm Inn}_a$ be the 
inner automorphism of $H$ defined by $a$, this means that ${\rm Inn}_a(h)=aha^{-1}$, 
for all $h\in H$. 

\begin{corollary}
If $H$ is a finite dimensional unimodular quasi-Hopf algebra then 
$S^4={\rm Inn}_{S(\un{g})}$, where $\un{g}$ is the modular element of $H$. 
Furthermore, if ${\cal L}={\cal R}$ then $S^4$ is the identity morphism of $H$. 
\end{corollary}

\begin{proof}
If $H$ is unimodular, then $\mu=\va$ and $f_{\mu}=1$ and if, moreover, 
${\cal L}={\cal R}$ then $\un{g}=1$, cf. \coref{modularforunimod}. 
\end{proof} 

We have seen how the antipode of $H^*$, or its inverse, acts on left or right cointegrals.
Let us now investigate how the antipode of $H$ acts on left or right integrals in $H$.

\begin{proposition}\prlabel{SinvSint}
If $t, r$ are non-zero left, respectively right, integrals in $H$ then 
$$\begin{array}{cc}
S(t)=\mu(\b)^{-1}\mu(q^2t_2p^2)q^1t_1p^1;&
S^{-1}(t)=\mu^{-1}(\un{g})\mu(q^2t_2p^2)q^1t_1p^1;\\
S(r)=(\mu^{-1}(\un{g})\mu(\a\b))^{-1}\mu^{-1}(q^2r_2p^2)q^1r_1p^1;&
S^{-1}(r)=\mu(\a)^{-1}\mu^{-1}(q^2r_2p^2)q^1r_1p^1.
\end{array}$$
\end{proposition}

\begin{proof}
Consider $\l\in {\cal L}$ such that $\l(S^{-1}(t))=1$, so that $\mu(\b)\l(t)=1$. 
If we define $r'=\mu(t_2p^2)t_1p^1$ then 
\begin{eqnarray*}
&&\hspace*{-1cm}
r'h=\mu(t_2p^2)t_1p^1h\equal{\equref{qr1}}\mu(t_2h_{(1, 2)}p^2S(h_2))
t_1h_{(1, 1)}p^1\\
&&\hspace*{2cm}
\equal{\equref{gdim}}\mu(h_1)\mu(t_2p^2)\mu^{-1}(h_2)t_1p^1=\va(h)r',
\end{eqnarray*}
for all $h\in H$. Thus $r'$ is a right integral in $H$. As ${\rm dim}_k\int_r^H=1$ 
there exist $c, c'\in k$ such that $S(t)=cr'$ and $S^{-1}(t)=c'r'$. 
We have that $\l(S(t))=(S^{-1}(\un{g}^{-1})\rh \l)(t)=\mu^{-1}(\un{g}^{-1})\l(t)=
(\mu^{-1}(\un{g})\mu(\b))^{-1}$ and 
\[
\l(r')=\mu(t_2p^2)\l(t_1p^1)\equal{\equref{f2}}
\mu(S^{-1}(\b)q^2t_2p^2)\l(q^1t_1p^1)
\equal{\equref{fu5}}\mu^{-1}(\b)\mu^{-1}(\un{g}^{-1}).
\]
It follows that $c=(\mu(\b)\mu^{-1}(\b))^{-1}$ and $c'=\mu^{-1}(\b \un{g}^{-1})^{-1}$. 
The formulas for $S(t)$ and $S^{-1}(t)$ now follow from these relations and \equref{f2}.\\
Let $r$ be a right integral, and let $t=S(r)$. Take $\l\in {\cal L}$ such 
that $\l(r)=\l(S^{-1}(t))=1$. Proceeding as in the first part of the proof, we can show that
$t':=\mu^{-1}(q^2r_2)q^1r_1$ is a left integral, hence there exist 
$b, b'\in k$ such that $S(r)=bt'$ and $S^{-1}(r)=b't'$. The formula for 
$S^{-1}(r)$ can be obtained by applying the formula for $S(t)$ to $H^{\rm op}$.\\
As the formula for $S^{-1}(t)$ contains $\un{g}$ and we do not have an analogue 
of $\un{g}$ in $H^{\rm op}$ we cannot derive the formula for $S(r)$ from the one of 
$S^{-1}(t)$. Nevertheless, we can obtain it by computing $b$ as follows. We  have
\begin{eqnarray*}
\l(S(r))&=&(S^{-1}(\un{g}^{-1})\rh \l)(r)=\va(\un{g}^{-1})\l(r)=\l(t)=\mu(\b)^{-1};\\
\l(t')&=&\l(q^1r_1)\mu^{-1}(q^2r_2)
\equal{(*)}\l(q^1r_1p^1)\mu^{-1}(q^2r_2p^2\a)\\
&\equal{\equref{qrpversusqtp}}&\mu(\tqla)\l(\tqlb S^{-1}(q^2t_2p^2))\mu^{-1}(S^{-1}(q^1t_1p^1)\a)\\
&=&\mu(\tqla)\phi(q^2t_2p^2S(\tqlb))\mu(q^1t_1p^1)\mu^{-1}(\a).
\end{eqnarray*}    
At (*), we used \equref{f2}, applied to $H^{\rm op}$, and 
$\phi=\l\circ S^{-1}$ is the Frobenius morphism of $H$. 
We now use the Nakayama automorphism $\chi$ of $H$ associated to $\phi$ to 
 compute 
\begin{eqnarray*}
\l(t')&=&\mu(\tqla)\phi(\chi(S(\tqlb))q^2t_2p^2)\mu(q^1t_1p^1)\mu^{-1}(\a)\\
&\equal{\equref{f1}}&\mu(\tqla)\phi(q^2t_2p^2)\mu(S(\chi(S(\tqlb)))q^1t_1p^1)\mu^{-1}(\a)\\
&=&\mu(\tqla)\l(S^{-1}(q^2t_2p^2))\mu^{-1}(S^{-1}(q^1t_1p^1))\mu^{-1}(\chi(S(\tqlb)))\mu^{-1}(\a)\\
&\equal{\equref{fu5}}&\mu(\tqla)\mu^{-1}(\un{g})\mu^{-1}(S_{\mu}(\tqlb))\mu^{-1}(\a).
\end{eqnarray*}
In the last equality, we used the (obvious) identity $\chi\circ S=S^2\circ S_{\mu}$,
and we used the notation $S_{\mu}(h):=S(h)\lh \mu$. Now we have for all $h\in H$ that
$\mu^{-1}(S_{\mu}(h))=\mu(S(h)_1)\mu^{-1}(S(h)_2)=\va(S(h))=\va(h)$, so
$\l(t')=\mu(\a)\mu^{-1}(\un{g}\a)$, and therefore 
$\mu(\b)^{-1}=b \mu(\a)\mu^{-1}(\un{g}\a)$. Hence
$$
S(r)=\mu(\a\b)^{-1}\mu^{-1}(\un{g}\a)^{-1}\mu^{-1}(q^2r_2)q^1r_1
=(\mu^{-1}(\un{g})\mu(\a\b))^{-1}\mu^{-1}(q^2r_2p^2)q^1r_1p^1,
$$
and this completes the proof.
\end{proof}

Applying \prref{SinvSint} and \equref{mumuinv}, we have the following result.

\begin{corollary}\colabel{Ssqint}
Let $H$ be a finite dimensional quasi-Hopf algebra. For all  $t\in \int_l^H$ and $r\in \int_r^H$,
we have that
\[
S^2(t)=(\mu^{-1}(\un{g})\mu(\b))^{-1}t~~{\rm and}~~
S^2(r)=(\mu^{-1}(\un{g})\mu(\b))^{-1}r.
\]
\end{corollary}

%%%%%%%%%%%%%%%%%%%%%%%%%%%%%%%%%%%%%%%%%%%%%%%%%%%%%%%%%%%%%
\section{The (co)integrals of a quantum double}
%%%%%%%%%%%%%%%%%%%%%%%%%%%%%%%%%%%%%%%%%%%%%%%%%%%%%%%%%%%%%%
\setcounter{equation}{0}
We are now able to provide explicit formulas for the integral in and the cointegral on
the quantum double $D(H)$ of a finite dimensional quasi-Hopf algebra $H$. 
In particular, we will see that the conjecture 
of Hausser and Nill \cite{hn3} holds if $H$ is unimodular, but not in general.

We first recall the definition and properties of $D(H)$. In the sequel,
$\{e_i\}_i$ will be a basis of $H$, and $\{e^i\}_i$ the corresponding dual basis of $H^*$.
$\Omega=\Omega^1\ot \cdots\ot \Omega^5\in H^{\ot 5}$ is defined by
\begin{equation}\eqlabel{ome}
\Omega =X^1_{(1, 1)}y^1x^1\ot X^1_{(1, 2)}y^2x^2_1\ot X^1_2y^3x^2_2\ot
\smi (f^1X^2x^3)\ot \smi (f^2X^3)~,
\end{equation}
where $f=f^1\ot f^2$ is the twist of Drinfeld defined in \equref{f}. As a vector space,
$D(H)=H^*\ot H$. The multiplication is given by the formula
\begin{equation}\eqlabel{mdd}
(\varphi \Join h)(\psi \Join h^{'})=
[(\Omega ^1 \rh \varphi \lh \Omega ^5)
(\Omega ^2h_{(1, 1)}\rh \psi \lh \smi (h_2)\Omega ^4]
\Join \Omega ^3h_{(1, 2)}h^{'},
\end{equation}
for all $\varphi , \psi \in H^*$ and $h, h^{'}\in H$, where 
we wrote $\varphi\Join h$ in place of $\varphi\ot h$ in order to 
distinguish this new multiplication on $H^*\ot H$. 
The unit of $D(H)$ is $\va \Join 1$. The explicit formulas for the 
comultiplication, counit, reassociator and antipode are
\begin{eqnarray}
&&\Delta _D(\varphi \Join h)=(\va \Join X^1Y^1)(p^1_1x^1\rh \varphi _2\lh
Y^2\smi (p^2)\Join p^1_2x^2h_1)\nonumber\\
&&\hspace*{2cm}\ot (X^2_1\rh \varphi_1\lh \smi (X^3)\Join X^2_2Y^3x^3h_2),\eqlabel{cdd}\\
&&\va _D(\varphi \Join h)=\va (h)\varphi (\smi (\a )),\eqlabel{codd}\\
&&\Phi_D=(\va \Join X^1)\ot (\va \Join X^2)\ot (\va\Join X^3),\\
&&S_D(\varphi \Join h)=(\va \Join S(h)f^1)(p^1_1U^1\rh \ov{S}^{-1}(\varphi)\lh
f^2\smi (p^2)\Join p^1_2U^2),\eqlabel{add}\\
&&\a _D=\va \Join \a~~;~~
\b _D=\va \Join \b .\eqlabel{abdd}
\end{eqnarray}
Here $p_R=p^1\ot p^2$, $f=f^1\ot f^2$ and $U=U^1\ot U^2$
are the elements defined by \equref{qr}, \equref{f} and \equref{elemsUandV}. 
$D(H)$ is a quasi-Hopf algebra and $H$ is a quasi-Hopf subalgebra via the 
canonical morphism $i_D:\ H\ra D(H)$, $i_D(h)=\va \Join h$.

Now we take a non-zero left cointegral $\l$ on $H$ and  a non-zero 
right integral $r$ in $H$, and define $T:=\mu^{-1}(\d^2)\d^1\rh \l\in H^*$, where 
$\mu$ is the modular element of $H^*$ and $\d=\d^1\ot \d^2$ is the element 
of $H\ot H$ defined in \equref{delta}. We claim that $T\Join r$ is a non-zero 
left and right integral in $D(H)$. That it is non-zero follows easily from the fact 
that $r$ is non-zero and $T(r)=\mu^{-1}(\d^2)\l(r\d^1)=
\mu^{-1}(\va(\d^1)\d^2)\l(r)=\mu^{-1}(\b)\l(r)\not=0$, as 
${\cal L}\times \int_r^H\ni (\l', r')\mapsto \l'(r')\in k$ is non-degenerated, 
see \cite[Lemma 4.4]{hn3}. 

The difficult part is to show that $T\Join r$ is a left and right integral in 
$D(H)$. To prove that it is a left integral we need the following formulas,
\begin{eqnarray}
&&\Delta(h_1)\delta (S\ot S)(\Delta^{\rm cop}(h_2))=\va(h)\delta,~~
{\rm for~all}~~h\in H,\eqlabel{fdeltaDrinf}\\
&&Y^1\d^1S(Y^3_2)\ot Y^2\d^2S(Y^3_1)=\b S(\tplb)\ot S(\tpla),\eqlabel{foressleftintqd}\\
&&z^1\tpla \ot z^2\tilde{p}^2_1\ot z^3\tilde{p}^2_2=
Y^2_1Z^2S^{-1}(Y^1Z^1\b)\ot Y^2_2Z^3\ot Y^3,\eqlabel{foressleftintqd2}\\
&&X^1\ot q^1X^2_1\ot S^{-1}(X^3)q^2X^2_2=q^1_1x^1\ot q^1_2x^2\ot q^2x^3.
\eqlabel{foressleftintqd3}
\end{eqnarray}
\equref{fdeltaDrinf} can be found in \cite{dri}, and (\ref{eq:foressleftintqd}-
\ref{eq:foressleftintqd3})
can be deduced easily from the definitions of $\delta$, $p_L$ and $q_R$, 
and the axioms \equref{q3} and \equref{q5}. We leave the details to the reader.

\begin{proposition}
Keeping the above setting and notation, we have that $T\Join r$ is a left integral in $H$. 
\end{proposition}

\begin{proof}
We check this assertion by direct computation. If $\v\in H^*$ and $h\in H$ then 
\begin{eqnarray*}
&&\hspace*{-1.5cm}(\v\Join h)(T\Join r)\\
&=&\mu^{-1}(\d^2)(\Omega^1\rh \v\lh \Omega^5)
(\Omega^2h_{(1, 1)}\d^1\rh \lambda \lh S^{-1}(h_2)\Omega^4)\Join \Omega^3h_{(1, 2)}r\\
&\equal{\equref{gdim}}&\mu^{-1}(\Omega^3h_{(1, 2)}\d^2)
(\Omega^1\rh \v\lh \Omega^5)
(\Omega^2h_{(1, 1)}\d^1\rh \lambda \lh S^{-1}(h_2)\Omega^4)\Join r\\
&\equal{\equref{f4}}&\mu^{-1}(\Omega^3h_{(1, 2)}\d^2S((h_2)_1))
(\Omega^2h_{(1, 1)}\d^1S((h_2)_2)\rh \lambda \lh \Omega^4)\Join \Omega^3h_{(1, 2)}r\\
&\equal{\equref{fdeltaDrinf}}&\va(h)\mu^{-1}(\Omega^3\d^2)
(\Omega^1\rh \v\lh \Omega^5)(\Omega^2\d^1\rh \lambda \lh \Omega^4)\Join r.
\end{eqnarray*}
Therefore it suffices to show that 
\[
\mu^{-1}(\Omega^3\d^2)\l(\Omega^4h_2\Omega^2\d^1)\Omega^5h_1\Omega^1=
\mu^{-1}(\d^2)\l(h\d^1)S^{-1}(\a),
\]
for all $h\in H$. To this end, we compute
\begin{eqnarray*}
&&\hspace*{-2cm}
\mu^{-1}(\Omega^3\d^2)\l(\Omega^4h_2\Omega^2\d^1)\Omega^5h_1\Omega^1\\
&\equal{\equref{ome}}&\hspace*{-3mm}
\mu^{-1}(X^1_2y^3x^2_2\d^2)\l(S^{-1}(f^1X^2x^3)h_2X^1_{(1, 2)}y^2x^2_1\d^1)\\
&&\hspace*{5mm}
S^{-1}(f^2X^3)h_1X^1_{(1, 1)}y^1x^1\\
&\equal{(\ref{eq:q3},\ref{eq:foressleftintqd})}&
\hspace*{-3mm}
\mu^{-1}(X^1_2y^2x^3_1S(\tpla))\l(S^{-1}(f^1X^2y^3x^3_2)h_2X^1_{(1, 2)}y^1_2x^2\b S(\tplb))\\
&&\hspace*{5mm}
S^{-1}(f^2X^3)h_1X^1_{(1, 1)}y^1_1x^1\\
&\equal{(\ref{eq:f4},\ref{eq:ql1})}&
\hspace*{-3mm}
\mu^{-1}(X^1_2y^2S(\tpla))\l(S^{-1}(f^1X^2y^3)h_2X^1_{(1, 2)}y^1_2x^2\b S(\tplb x^3))\\
&&\hspace*{5mm}
S^{-1}(f^2X^3)h_1X^1_{(1, 1)}y^1_1x^1\\
&\equal{(\ref{eq:qr},\ref{eq:pplu})}&
\hspace*{-3mm}
\mu^{-1}(X^1_2y^2S(\tpla))\l (S^{-1}(f^1X^2y^3)(hX^1_1y^1S(\tPla))_2U^2S(\tplb))\\
&&\hspace*{5mm}
S^{-1}(f^2X^3)(hX^1_1y^1S(\tPla))_1U^1\tPlb\\
&\equal{(\ref{eq:f4},\ref{eq:fu1})}&
\hspace*{-3mm}
\mu^{-1}(X^1_2y^2S(\tpla))\mu(X^2_1y^3_1)
\le \l, S^{-1}(f^1)(hX^1_1y^1S(X^2_{(2, 1)}y^3_{(2, 1)}\tilde{p}^2_1\tPla))_2U^2\ri\\
&&\hspace*{5mm}
S^{-1}(f^2X^3)(hX^1_1y^1S(X^2_{(2, 1)}y^3_{(2, 1)}\tilde{p}^2_1\tPla))_1U^1
X^2_{(2, 2)}y^3_{(2, 2)}\tilde{p}^2_2\tPlb\\
&\equal{\equref{altcharactleftcoint}}&\hspace*{-3mm}
\mu^{-1}(X^1_2y^2S(\tpla))\mu(X^2_1y^3_1)\mu(q^1_1z^1)
\le \l, hX^1_1y^1S(q^1_2z^2X^2_{(2, 1)}y^3_{(2, 1)}\tilde{p}^2_1\tPla)\ri\\
&&\hspace*{5mm}
S^{-1}(X^3)q^2z^3X^2_{(2, 2)}y^3_{(2, 2)}\tilde{p}^2_2\tPlb\\
&\equal{\equref{q1}}&\hspace*{-3mm}
\mu^{-1}(X^1_2y^2)\mu(q^1_1X^2_{(1, 1)}y^3_{(1, 1)}z^1\tpla) 
\le \l, hX^1_1y^1S(q^1_2X^2_{(1, 2)}y^3_{(1, 2)}z^2\tilde{p}^2_1\tPla\ri \\
&&\hspace*{5mm}
S^{-1}(X^3)q^2X^2_2y^3_2z^3\tilde{p}^2_2\tPlb\\
&\equal{(\ref{eq:foressleftintqd2},\ref{eq:foressleftintqd3})}&
\hspace*{-3mm} 
\mu((q^1_2)_1(x^2y^3_1Y^2)_1Z^2S^{-1}(q^1_{(1, 2)}x^1_2y^2Y^1Z^1\b))\\
&&\hspace*{5mm}
\le \l , hq^1_{(1, 1)}x^1_1y^1S((q^1_2)_2(x^2y^3_1Y^2)_2Z^3\tPla)\ri 
q^2x^3y^3_2Y^3\tPlb\\
&\equal{(\ref{eq:q3},\ref{eq:q1},\ref{eq:q5})}&\hspace*{-3mm}
\mu((q^1_2z^3)_1Z^2S^{-1}(q^1_{(1, 2)}z^2Z^1\b))\\
&&\hspace*{5mm}
\le \l , hq^1_{(1, 1)}z^1y^1S((q^1_2z^3)_2Z^3y^2\tPla)\ri 
q^2y^3\tPlb\\
&\equal{\equref{q1}}&\hspace*{-3mm} 
\mu^{-1}(z^2(q^1_2)_1Z^1\b S(z^3_1(q^1_2)_{(2, 1)}Z^2))\\
&&\hspace*{5mm}
\le \l, hz^1q^1_1y^1S(z^3_2(q^1_2)_{(2, 2)}Z^3y^2\tPla)\ri q^2y^3\tPlb\\
&\equal{(\ref{eq:q1},\ref{eq:q5})}&
\hspace*{-3mm}
\mu^{-1}(z^2Z^1\b S(z^3_1Z^2))\le \l, hz^1q^1_1y^1S(z^3_2Z^3q^1_2y^2\tPla)\ri 
q^2y^3\tPlb\\
&\equal{(\ref{eq:ql},\ref{eq:q5})}&
\hspace*{-3mm}
\mu^{-1}(z^2Z^1\b S(z^3_1Z^2))\le \l, hz^1\b S(z^3_2Z^3)\ri S^{-1}(\a)\\
&\equal{\equref{delta}}&\hspace*{-3mm}
\mu^{-1}(\d^2)\l(h\d^1)S^{-1}(\a),
\end{eqnarray*}
for all $h\in H$, and  this finishes the proof. 
\end{proof}

\begin{corollary}\colabel{DHss}
The quantum double $D(H)$ is a semisimple algebra if and only if $H$ is semisimple 
and admits a normalized left cointegral, that is a left cointegral $\l$ satisfying 
$\l(S^{-1}(\a)\b)\not=0$. 
\end{corollary}

\begin{proof}
This is an immediate consequence of the Maschke theorem for quasi-Hopf algebras proved in 
\cite{pan}. Note that, for the non-zero left integral 
$\mathbb{T}=\mu^{-1}(\d^2)\d^1\rh \l\Join r$ in $D(H)$ we have $\va_D(\mathbb{T})=\va(r)\mu^{-1}(\d^2)\l(S^{-1}(\a)\d^1)$, 
and so $\va_D(\mathbb{T})\not=0$ if and only if $\va(r)\not=0$ and $\mu^{-1}(\d^2)\l(S^{-1}(\a)\d^1)\not=0$. 
But $\va(r)\not=0$ implies $H$ semisimple, and therefore unimodular, in which case $\mu^{-1}(\d^2)\d^1=\b$.     
\end{proof}

\begin{examples}
1) $D(H(2))$ is semisimple because $H(2)$ is semisimple 
and the left cointegral $P_g$ on $H(2)$ found in \exref{lcointh2} satisfies 
$P_g(S^{-1}(g))=P_g(g)=1$.\\
2) $D(H_{\pm}(8))$ is not semisimple because $H_{\pm}(8)$ is not semisimple, cf. 
\exsref{exsinteginQHAs} 2).  
\end{examples}

From \cite[Theorem 6.5]{btfact} we know that $\mathbb{T}=T\Join r$ is a 
right integral in $D(H)$. In \seref{app}, we will present a direct proof, and we will
also show that 
\begin{eqnarray}
&&\hspace*{-5mm}
\l(S^{-1}(f^2)h_1g^1S(h'))S^{-1}(f^1)h_2g^2=\mu(f^1)\mu^{-1}(U^2_2{\bf U}^2\a)\mu(\b)\nonumber\\
&&\hspace*{1cm}
\mu(U^1y^1_2x^2)\l(hS(y^3x^3_2h'_2\tplb))S^{-1}(\un{g}^{-1}y^1_1x^1)S(S(U^2_1{\bf U}^1y^2x^3_1h'_1\tpla)f^2),
\eqlabel{normdefmodelem}
\end{eqnarray}
for all $\lambda\in {\cal L}$ and $h, h'\in H$, where $U=U^1\ot U^2={\bf U}^1\ot {\bf U}^2$ is the 
element defined in \equref{elemsUandV}. 

Now we will focus on the cointegrals on $D(H)$. In the sequel, we will identify
$D(H)^*\cong H\ot H^*$.

\begin{proposition}
Take non-zero elements $\l\in {\cal L}$ and $r\in \int_r^H$. Then 
\begin{equation}\eqlabel{leftcointDH}
\Gamma = r\Join \mu(\tpla)S(\tplb)\rh \l\lh \mu^{-1}(f^1)S^{-1}(f^2)\in D(H)^*
\end{equation}
 is a non-zero left cointegral on $D(H)$.
\end{proposition}

\begin{proof}
It is clear that $\Gamma\neq 0$. 
Since $H$ is a quasi-Hopf subalgebra of $D(H)$ via $i_D$ it follows that the 
elements $U$ and $V$ for $D(H)$ are $U_D=\va\Join U^1\ot \va\Join U^2$ and 
$V_D=\va\Join V^1\ot \va\Join V^2$. Applying \equref{foressleftintqd3} to $H^{\rm op}$,
we have
\begin{equation}\eqlabel{peq}
X^1p^1_1\ot X^2p^1_2\ot X^3p^2=x^1\ot x^2_1p^1\ot x^2_2p^2S(x^3).
\end{equation}
Identifying $D(H)^*\cong H\ot H^*$, we compute
\begin{eqnarray*}
&&\hspace*{-2cm}
\Gamma((\va\Join V^2)(X^2_1\rh \v_1\lh S^{-1}(X^3)\Join X^2_2Y^3x^3h_2)(\va\Join U^2))\\
&&\hspace*{-1cm}
(\va\Join V^1)(\va\Join X^1Y^1)(p^1_1x^1\rh \v_2\lh Y^2S^{-1}(p^2)\Join p^1_2x^2h_1)(\va\Join U^1)\\
&\equal{\equref{mdd}}&\v_1(S^{-1}(V^2_2X^3)rV^2_{(1, 1)}X^2_1)
\l(S^{-1}(f^2)V^2_{(1, 2)}X^2_2Y^3x^3h_2U^2S(\tplb))\mu(\tpla)\\
&&\mu^{-1}(f^1)(\va \Join V^1X^1Y^1)(p^1_1x^1\rh \v_2\lh Y^2S^{-1}(p^2)\Join p^1_2x^2h_1U^1)\\
&=&\v(r)\mu(S^{-1}(f^1)V^2_2X^3)\mu(\tpla)\l(S^{-1}(f^2)V^2_1X^2h_2U^2S(\tplb))\\
&&\va\Join V^1X^1h_1U^1\\
&\equal{(\ref{eq:elemsUandV},\ref{eq:ca})}&
\v(r)\mu^{-1}(S(X^3)F^1f^1_1p^1_1)\mu(\tpla)\l(S^{-1}(S(X^2)F^2f^1_2p^1_2)h_2U^2S(\tplb))\\
&&\va\Join S^{-1}(S(X^1)f^2p^2)h_1U^1\\
&\equal{(\ref{eq:pf},\ref{eq:peq})}&
\v(r)\mu^{-1}(f^1x^1)\mu(\tpla)\l(S^{-1}(F^1f^2_1x^2_1p^1)h_2U^2S(\tplb))\\
&&\va\Join S^{-1}(F^2f^2_2x^2_2p^2S(x^3))h_1U^1\\
&\equal{(\ref{eq:ca},\ref{eq:elemsUandV},\ref{eq:fu1})}&
\v(r)\mu^{-1}(f^1x^1)\mu(\tpla)\l(V^2(S^{-1}(f^2x^2)hS(\tilde{p}^2_1))_2U^2)\\
&&\va\Join x^3V^1(S^{-1}(f^2x^2)hS(\tilde{p}^2_1))_1U^1\tilde{p}^2_2\\
&\equal{(\ref{eq:charactleftcoint},\ref{eq:f4})}&
\v(r)\mu^{-1}(f^1x^1)\mu(x^2_1y^1\tpla)\l(S^{-1}(f^2)hS(x^2_2y^2\tilde{p}^2_1))\va\Join x^3y^3\tilde{p}^2_2\\
&\equal{(\ref{eq:ql},\ref{eq:q3})}&
\v(r)\mu(\tpla)\mu^{-1}(f^1)\l(S^{-1}(f^2)hS(\tplb))\va\Join 1=\Gamma(\v\Join h)\va \Join 1.
\end{eqnarray*} 
As $D(H)$ is unimodular the above computation shows that $\Gamma$ is a left cointegral on $D(H)$, as desired.
\end{proof}

\begin{corollary}
$D(H)$ admits a normalized left cointegral if and only if $D(H)$ is a semisimple algebra.
\end{corollary}

\begin{proof}
For the non-zero left cointegral $\Gamma$ defined in \equref{leftcointDH} we have 
\[
\Gamma(S_D^{-1}(\va\Join \a)(\va\Join \b))=\Gamma(\va\Join S^{-1}(\a)\b)=\va(r)\mu(\tpla)\mu^{-1}(f^1)
\l(S^{-1}(\a f^2)\b \tplb). 
\]
This scalar is non-zero if and only if $\va(r)\not=0$ and $\mu(\tpla)\mu^{-1}(f^1)
\l(S^{-1}(\a f^2)\b \tplb)\not=0$. But, as we have already mentioned, $\va(r)\not=0$ implies $H$ unimodular, and 
in this case $\mu(\tpla)\mu^{-1}(f^1)\l(S^{-1}(\a f^2)\b \tplb)=\l(S^{-1}(\a)\b)$. Then the result follows 
from \coref{DHss}.
\end{proof}

Now we describe the space of right cointegrals on $D(H)$.

\begin{proposition}\prlabel{rcointDH}
If $t\in \int_l^H$ and $\l\in {\cal L}$ are non-zero then $t\Join \l\circ S$ is a non-zero right 
cointegral on $D(H)$. 
\end{proposition}

\begin{proof}
Since $D(H)$ is unimodular $\Gamma\circ S_D$ is a non-zero right cointegral on $D(H)$, cf. 
\coref{antipcointunimod}. 
So it suffices to show that $\Gamma\circ S_D=S(r)\Join \l\circ S$. 
Applying $\mu$ to the both sides of \equref{normdefmodelem} we obtain after a straightforward computation that 
\[
\mu(S^{-1}(f^1)h_2g^2)\l(S^{-1}(f^2)h_1g^1S(h'))=\mu^{-1}(\a\un{g}^{-1})\mu(\tqla)\mu(\tilde{q}^2_1h'_1\tpla)
\l(hS(\tilde{q}^2_2h'_2\tplb)),
\]
for all $h, h'\in H$. Consequently, by \equref{pqla} we obtain that 
\begin{eqnarray}
&&\hspace*{-2cm}
\mu^{-1}(\tqla)\mu(S^{-1}(f^1)S(h)_2g^2)\l(S^{-1}(f^1)S(h)_1g^1S(h'\tqlb))\nonumber\\
&&\hspace*{1cm}
=\mu^{-1}(\a \un{g}^{-1})\mu(\tqla)\mu(\tilde{q}^2_1h'_1)\l(S(\tilde{q}^2_2h'_2h)),\eqlabel{ffff}
\end{eqnarray} 
for all $h, h'\in H$. We then compute 
\begin{eqnarray*}
&&\hspace*{-15mm}
\Gamma\circ S_D(\v\Join h)=
\Gamma((\va\Join S(h)f^1)(p^1_1U^1\rh \v\circ S^{-1}\lh f^2S^{-1}(p^2)\Join p^1_2U^2))\\
&=&\v\circ S^{-1}(f^2S^{-1}(S(h)_2f^1_2p^2)rS(h)_{(1, 1)}f^1_{(1, 1)}p^1_1U^1)\\
&&\hspace*{1cm}\mu(\tpla)\mu^{-1}(F^1)
\l(S^{-1}(F^2)S(h)_{(1, 2)}f^1_{(1, 2)}p^1_2U^2S(\tplb))\\
&\equal{\equref{fpformula}}&\v(S^{-1}(r))\mu(S^{-1}(F^1)S(h)_2g^2S(\tqla))\mu(\tpla)
\l(S^{-1}(F^2)S(h)_1g^1S(\tplb \tqlb))\\
&\equal{\equref{ffff}}&
\mu^{-1}(\a \un{g}^{-1})\v(S^{-1}(r))\mu((\tqlb \tplb)_1)\mu (\tqla \tpla)\l(S((\tqlb \tplb)_2h))\\
&\equal{\equref{ql}}&\mu^{-1}(\un{g}^{-1})\v(S^{-1}(r))\mu^{-1}(\a)\mu(\a)\mu^{-1}(\b)\l(S(h))\\
&\equal{\equref{mumuinv}}&\v((\mu^{-1}(\un{g})\mu(\b))^{-1}S^{-1}(r))\l(S(h))
=(S(r)\Join \l\circ S)(\v\Join h),
\end{eqnarray*}
for all $\v\in H^*$ and $h\in H$, where in the last equality we used \coref{Ssqint}. 
\end{proof}

The modular element of $D(H)^*$ is $\mu_D=\va_D$. Our next aim is to compute the modular element 
$\un{g}_D$ of $D(H)$. 
To this end, we will need an explicit formula for the inverse of the antipode $S_D$ of $D(H)$
 and a lemma.
 
\begin{lemma}
The composition inverse $S_D^{-1}$ of the antipode of $D(H)$ is given by the formula
\[
S_D^{-1}(\v\Join h)=(\va\Join S^{-1}(f^2h))(p^1_1S^{-1}(q^2g^2)\rh \v\circ S\lh S^{-1}(p^2f^1)\Join p^1_2S^{-1}(q^1g^1)),
\]
for all $\v\in H^*$ and $h\in H$. 
\end{lemma}

\begin{proof}
We first observe that \equref{qr1} and \equref{pqr} imply that
\[
(\va\Join q^1h_1)(p^1_1\rh \v \lh q^2h_2S^{-1}(p^2)\Join p^1_2)=h_1\rh \v\Join h_2,
\]
for all $\v\in H^*$ and $h\in H$. Consequently,
\begin{eqnarray*}
&&\hspace*{-2cm}
(\va \Join q^1S(P^1)_1)(p^1_1U^1P^2f^1\rh \v \lh q^2S(P^1)_2S^{-1}(p^2)\Join p^1_2U^2)(\va\Join f^2)\\
&=&(S(P^1)_1U^1P^2f^1\rh \v \Join S(P^1)_2U^2)(\va\Join f^2)\\
&\equal{(\ref{eq:elemsUandV},\ref{eq:ca})}&
(g^1S(q^2P^1_2)P^2f^1\rh \v\Join g^2S(q^1p^1_1))(\va\Join f^2)\equal{\equref{pqra}}\v\Join 1, 
\end{eqnarray*} 
for all $\v\in H^*$. By the definition of $S_D$ we have 
\[
S_D(S^{-1}(g^2)\rh \v\Join S^{-1}(g^1))=p^1_1U^1\rh \v\circ S^{-1}\lh S^{-1}(p^2)\Join p^1_2U^2,
\]
and combining these two relations we find for all $\v \in H^*$ that
\[
(\va\Join q^1S(P^1)_1)S_D(S^{-1}(g^2)\rh (P^2f^1\rh \v \lh q^2S(P^1)_2)\circ S\Join S^{-1}(g^1))(\va\Join f^2)
\]
equals $\v\Join 1$. As $H$ is a quasi-Hopf subalgebra of $D(H)$ it follows that 
$S_D^{-1}(\va\Join h)=\va\Join S^{-1}(h)$, for all $h\in H$. 
This and the fact that $S_D^{-1}$ is an anti-algebra 
morphism leads us to the equality
\begin{eqnarray*}
&&\hspace*{-1cm}
S_D^{-1}(\v\Join h)
=S_D^{-1}(\va\Join h)S_D^{-1}(\v\Join 1)\\
&=&(\va\Join S^{-1}(f^2h))\left(S^{-1}(q^2S(P^1)_2g^2)\rh \v\circ S\lh S^{-1}(P^2f^1)
\Join S^{-1}(q^1S(P^1)_1g^1)\right)\\
&\equal{\equref{ca}}&(\va\Join S^{-1}(f^2h))(P^1_1S^{-1}(q^2g^2)\rh \v\circ S\lh S^{-1}(P^2f^1)\Join P^1_2S^{-1}(q^1g^1)),
\end{eqnarray*} 
for all $\v\in H^*$ and $h\in H$.
\end{proof}

\begin{lemma}
Let $H$ be a finite dimensional quasi-Hopf algebra. Then for all $r\in \int_r^H$ and $h\in H$ the following 
equalities hold,
\begin{eqnarray}
&&hr_1\ot r_2=\mu^{-1}(h_1p^1)q^1r_1\ot S^{-1}(h_2p^2)q^2r_2,\eqlabel{rint3}\\
&&r_1U^1\ot r_2U^2S(h)=r_1U^1h\ot r_2U^2,\eqlabel{rint4}\\
&&V^1r_1\ot S^{-1}(h)V^2r_2=\mu(h_1)h_2V^1r_1\ot V^2r_2.\eqlabel{rint5}
\end{eqnarray}
\end{lemma}

\begin{proof}
\equref{rint3} follows since 
\begin{eqnarray*}
&&\hspace*{-20mm}
hr_1\ot r_2\equal{\equref{pqra}}hq^1p^1_1r_1\ot S^{-1}(p^2)q^2p^1_2r_2\\
&\equal{\equref{qr1a}}&q^1(h_1p^1r)_1\ot S^{-1}(h_2p^2)q^2(h_1p^1r)_2\\
&=&
\mu^{-1}(h_1p^1)q^1r_1\ot S^{-1}(h_2p^2)q^2r_2.
\end{eqnarray*}
\equref{rint4} is a direct consequence of \equref{fu1} while \equref{rint5} 
can be proved with the help of \equref{fv1}. We leave the verifications to the reader. 
\end{proof}

In order to compute the modular element $\un{g}_D$ of $D(H)$ 
we need a left integral $\mathbb{T}$ in $D(H)$ and a left cointegral 
$\Gamma$ on $D(H)$ such that $\Gamma(S_D^{-1}(\mathbb{T}))=1$. 
Since $\mu_D=\va_D$ it turns out that this is equivalent 
to $\Gamma(\mathbb{T})=1$. Also note that the unimodularity of $D(H)$ implies that $\Gamma\circ S_D=\Gamma\circ S_D^{-1}$.

Now take $\mathbb{T}=\mu^{-1}(\d^2)\rh \l'\Join r'$ for some $0\not=\l' \in {\cal L}$ and $0\not=r'\in \int_r^H$, and let 
$\Gamma$ be defined as in \equref{leftcointDH}. A simple inspection ensures that 
\[
\Gamma\circ S_D^{-1}(\mathbb{T})=\mu(\d^1)\mu^{-1}(\d^2)\l'(S(r))\l(S(r'))
\]
and since $S(\d^1)\a \d^2=S(\b)$ and $\va(\un{g})=\mu(\b)$, by \resref{invxi} (i) we conclude that 
$\Gamma\circ S_D^{-1}(\mathbb{T})=\mu^{-1}(\a)^{-1}\l'(S(r))\l(r')$. Thus we have to consider $\l, \l'$ and $r, r'$ such  that $\l'(S(r))\l(r')=\mu^{-1}(\a)$. 

\begin{proposition}
The modular element $\un{g}_D$ of $D(H)$ is given by 
\begin{eqnarray*}
\un{g}_D&=&\mu(g^1_1)\mu^{-1}(g^2)S_D^{-1}(\mu\Join g^1_2S^{-2}(\un{g}^{-1}))\\
&=&\mu(\tqla g^1)\mu^{-1}(\tpla)(\va\Join S^{-3}(\un{g}^{-1}))(\mu^{-1}\Join (S^{-1}(\tqlb g^2)\lh \mu^{-1})\tplb).
\end{eqnarray*}
\end{proposition}

\begin{proof}
Let $\l, \l'\in {\cal L}$ and $r, r'\in \int_r^H$ be such that $\l'(S(r))\l(r')=\mu^{-1}(\a)$. Then, by the above 
comments, \equref{fu5} and \equref{cdd}, the modular element $\un{g}_D$ can be computed as follows:
\begin{eqnarray*}
&&\hspace*{-15mm}
\un{g}_D
=\Gamma\circ S_D^{-1}((q^2_1X^2)_1\rh (\d^1\rh \l')_1\lh S^{-1}(X^3)\Join (q^2_1X^2)_2Y^3x^3r'_2P^2)\\
&&
\mu^{-1}(\d^2)
S_D^{-1}((\va\Join q^1X^1_1Y^1)(p^1_1x^1\rh (\d^1\rh \l')_2\lh Y^2S^{-1}(p^2)\Join p^1_2x^2r'_1P^1))\\
&=&\mu(q^2_1)\l(S(q^2_2Y^3x^3r'_2P^2))\l'(S(r))\mu(Y^2S^{-1}(p^2))\mu(p^1_1x^1)\\
&&\hspace*{5mm}
\mu^{-1}(\d^2)\mu(\d^1)S_D^{-1}((\va\Join q^1Y^1)(\mu\Join p^1_2x^2r'_1P^1))\\
&\equal{\equref{rint3}}&\mu^{-1}(p^1_{(2, 1)}x^2_1\mathbb{P}^1)\l'(S(r))\mu(q^2_1)
\l(S(q^2_2Y^3x^3S^{-1}(p^1_{(2, 2)}x^2_2\mathbb{P}^2)Q^2r'_2P^2))\\
&&\hspace*{5mm}
\mu^{-1}(\a)^{-1}\mu(\b)\mu(Y^2S^{-1}(p^2))\mu(p^1_1x^1)S_D^{-1}((\va \Join q^1Y^1)(\mu\Join Q^1r'_1P^1))\\
&\equal{(\ref{eq:peq},\ref{eq:q1})}&
\mu^{-1}(X^2S^{-1}(X^1\b))\mu^{-1}(\a)^{-1}\l'(S(r))\mu(Y^2S^{-1}(p^2))\mu(q^2_1)\\
&&\hspace*{5mm}\l(S(q^2_2Y^3S^{-1}(X^3p^1\b)Q^2r'_2P^2))S_D^{-1}((\va\Join q^1Y^1)(\mu\Join Q^1r'_1P^1))\\
&\equal{(\ref{eq:gdim},\ref{eq:UVpql})}&
\mu^{-1}(\a)^{-1}\l'(S(r))\mu(q^2_1)\mu(Y^2S^{-1}(p^2))\l(S(q^2_2Y^3S^{-1}(p^1\b)V^2r'_2U^2))\\
&&\hspace*{5mm}S_D^{-1}((\va\Join q^1Y^1)(\mu \Join V^1r'_1U^1))\\
&\equal{\equref{rint5}}&
\mu^{-1}(q^1_2Y^1_2p^2S(Y^2))\l'(S(r))\mu(q^2_1)
\l(S(q^2_2Y^3S^{-1}(q^1_1Y^1_1p^1\b)V^2r'_2U^2))\\
&&\hspace*{5mm}
\mu^{-1}(\a)^{-1}S_D^{-1}(\mu\Join V^1r'_1U^1)\\
&\equal{\equref{delta}}&
\mu^{-1}(q^1_2\d^2S(q^2_1))\l'(S(r))\l\circ S(S^{-1}(q^1_1\d^1S(q^2_2))V^2r'_2U^2)\\
&&\hspace*{5mm}\mu^{-1}(\a)^{-1}S_D^{-1}(\mu\Join V^1r'_1U^1)\\
&\equal{(\ref{eq:gdf},\ref{eq:ca})}&
\mu^{-1}((q^1\b S(q^2))_2g^2)\l'(S(r))\l\circ S(S^{-1}((q^1\b S(q^2))_1g^1)V^2r'_2U^2)\\
&&\hspace*{5mm}\mu^{-1}(\a)^{-1}S_D^{-1}(\mu\Join V^1r'_1U^1)\\
&\equal{(\ref{eq:q5},\ref{eq:rint5})}&
\mu(g^1_1)\mu^{-1}(g^2)\mu^{-1}(\a)^{-1}\l'(S(r))\l\circ S(V^2r'_2U^2)S_D^{-1}(\mu\Join g^1_2V^1r'_1U^1)\\
&\equal{\equref{charactleftcoint}}&
\mu(g^1_1)\mu^{-1}(g^2)S_D^{-1}(\mu\Join g^1_2S^{-2}(\un{g}^{-1})).
\end{eqnarray*}
In the second equality we used \prref{rcointDH} and the fact that 
$\Gamma\circ S_D^{-1}=\Gamma\circ S_D$, in the third one, we used
the properties $S(r)\in \int_l^H$ and $\mu$ is an algebra map, and in the last equality 
\resref{invxi} (i) and \equref{rint4}. We have also denoted by 
$\mathbb{P}^1\ot \mathbb{P}^2$ another copy of $p_R$. This proves the first formula 
for $\un{g}_D$. For the 
second one we use the form of $S_D^{-1}$ found above to compute
\begin{eqnarray*}
&&\hspace*{-20mm}
\un{g}_D
=\mu(g^1_1)\mu^{-1}(g^2)(\va\Join S^{-3}(\un{g}^{-1})S^{-1}(f^2g^1_2))\\
&&\hspace*{5mm} (p^1_1S^{-1}(q^2G^2)\rh \mu^{-1}\lh S^{-1}(p^2f^1)\Join p^1_2S^{-1}(q^1G^1))\\
&\equal{\equref{mdd}}&\mu^{-1}((S^{-1}(f^2g^1_2)_1p^1)_1S^{-1}(q^2G^2))\mu (S^{-1}(f^2g^1_2)_2p^2f^1)\mu(g^1_1)\\
&&\hspace*{5mm}\mu^{-1}(g^2)(\va\Join S^{-3}(\un{g}^{-1}))(\mu^{-1}\Join (S^{-1}((f^2g^1_2)_1p^1)_2S^{-1}(q^1G^1))\\
&\equal{(\ref{eq:ca},\ref{eq:pf})}&
\mu^{-1}(S^{-1}(f^2x^3g^1_{(2, 2)}\mathbb{G}^2)_1S^{-1}(q^2G^2))\mu^{-1}(g^2)\\
&&\hspace*{5mm}\mu(S^{-1}(F^2f^1_2x^2g^1_{(2, 1)}\mathbb{G}^1)\b F^1f^1_1x^1g^1_1)\\
&&\hspace*{5mm}(\va\Join S^{-3}(\un{g}^{-1}))(\mu^{-1}\Join S^{-1}(f^2x^3g^1_{(2, 2)}\mathbb{G}^2)_2S^{-1}(q^1G^1))\\
&\equal{(\ref{eq:fgab},\ref{eq:q1},\ref{eq:q5})}&
\mu^{-1}(S^{-1}(\a x^2\mathbb{G}^1)x^1)\mu^{-1}(S^{-1}(g^1x^3\mathbb{G}^2)_1S^{-1}(q^2G^2))\mu^{-1}(g^2)\\
&&\hspace*{5mm}(\va\Join S^{-3}(\un{g}^{-1}))(\mu^{-1}\Join S^{-1}(g^1x^3\mathbb{G}^2)_2S^{-1}(q^1G^1))\\
&\equal{(\ref{eq:ql},\ref{eq:ca},\ref{eq:formtplfversusqg})}&
\mu(\tqla \mathbb{G}^1)\mu(S(\tpla)f^2\tilde{q}^2_2\mathbb{G}^2_2G^2)\\
&&\hspace*{5mm}(\va\Join S^{-3}(\un{g}^{-1}))(\mu^{-1}\Join S^{-1}(S(\tplb)f^1\tilde{q}^2_1\mathbb{G}^2_1G^1))\\ 
&\equal{\equref{ca}}&\mu(\tqla\mathbb{G}^1)\mu^{-1}(\tpla)(\va\Join S^{-3}(\un{g}^{-1}))
(\mu^{-1}\Join (S^{-1}(\tqlb \mathbb{G}^2)\lh \mu^{-1})\tplb),
\end{eqnarray*}
and this completes the proof. 
\end{proof}

%%%%%%%%%%%%%%%%%%%%%%%%%%%%%%%%%%%%%%%%%%%%%%%%%
\section{Appendix}\selabel{app}
%%%%%%%%%%%%%%%%%%%%%%%%%%%%%%%%%%%%%%%%%%%%%%%%%
\setcounter{equation}{0}
We will present the proofs of \equref{normdefmodelem} and the fact that $T\Join r$ is a right integral in $D(H)$. 
\equref{normdefmodelem} can be viewed as the generalization to quasi-Hopf algebras
 of the well-known equality $\l(h_1)h_2=\l(h)\un{g}$, for $\l\in {\cal L}$ and $h\in H$, where $H$ is
a finite dimensional Hopf algebra. As we will see, it allows us to generalize Radford's
result \cite{radf} that the Drinfeld double of a finite dimensional Hopf algebra is
unimodular to quasi-Hopf algebras.

First we compute the inverse of the Nakayama isomorphism $\chi$ introduced in
\prref{leftFrobQHA}.

\begin{lemma}
$\chi^{-1}(h)=\mu(S^{-1}(uhu^{-1})_2)S^{-1}(S^{-1}(uhu^{-1})_1)$, for all $h\in H$, 
where $u$ is the element of $H$ introduced in \prref{deforuforint}. Consequently, 
\begin{equation}\eqlabel{inchileftcoint}
\mu^{-1}(\tqla h_1\tpla)\l \lh S^{-1}(\tqlb h_2\tplb)=\mu^{-1}(\a)\mu(\b)S(h)\rh \l,
\end{equation}
for all $\l\in {\cal L}$ and $h\in H$. 
\end{lemma}

\begin{proof}
It follows from \prref{deforuforint} that 
\begin{eqnarray*}
&&\hspace*{-2cm}
\l(\tqlb t_2\tplb)\tqla t_1\tpla =\Lambda(uS(\tqlb t_2\tplb))\tqla t_1\tpla\\
&\equal{\equref{ql1}}& \mu(S^{-1}(u)_2)\Lambda(S(\tqlb t_2\tplb))\tqla t_1\tpla S^{-1}(S^{-1}(u)_1).
\end{eqnarray*}
Hence, by \equref{NakAut} we obtain that
\begin{eqnarray*}
&&\hspace*{-2cm}
\chi^{-1}(h)=\phi(hS(q^2t_2p^2))q^1t_1p^1
=\l(q^2t_2p^2S^{-1}(h))q^1t_1p^1\\
&=&\l(\tqlb t_2\tplb S^{-1}(hu^{-1}))\tqla t_1\tpla\\
&\equal{\equref{ql1}}&\mu(S^{-1}(hu^{-1})_2)\l(\tqlb t_2\tplb)\tqla t_1\tpla S^{-1}(S^{-1}(hu^{-1})_1)\\
&=&\mu(S^{-1}(uhu^{-1})_2)S^{-1}(S^{-1}(uau^{-1})_1),
\end{eqnarray*}
as needed. In particular, this implies that
\begin{eqnarray*}
&&\hspace*{-20mm}
S^{-1}\chi^{-1}S^2(h)
=\mu(S^{-1}(u^{-1})_2S(h)_2S^{-1}(u)_2)S^{-2}(S^{-1}(u^{-1})_1S(h)_1S^{-1}(u)_1)\\
&=&\mu^{-1}(q^1_2g^2S(q^2))\mu(V^1)\mu(q^1_{(1, 2)}g^1_2S(V^2h)_2)S^{-2}(q^1_{(1, 1)}g^1_1S(V^2h)_1)\\
&\equal{(\ref{eq:ca},\ref{eq:pf},\ref{eq:q1})}&
\mu^{-1}(x^3G^2S(q^2X^1))\mu(V^1)\mu(x^2G^1S(V^2_1h_1X^2)f^2)\\
&&\hspace*{5mm}S^{-2}(x^1q^1S(V^2_2h_2X^3)f^1)\\
&\equal{(\ref{eq:elemsUandV},\ref{eq:ca},\ref{eq:pf})}&
\mu^{-1}(x^3G^2S(q^2X^1))\mu(S^{-1}(F^2Y^3p^2)y^1)\mu(x^2G^1S(y^2h_1X^2)F^1Y^2p^1_2)\\
&&\hspace*{5mm}S^{-2}(x^1q^1S(y^3h_2X^3)Y^1p^1_1)\\
&\equal{(\ref{eq:qr},\ref{eq:peq})}&
\mu(X^1\b)\mu(S^{-1}(F^2p^2)y^1)\mu(S(y^2h_1X^2)F^1p^1)S^{-1}(y^3h_2X^3)\\
&\equal{(\ref{eq:fgab},\ref{eq:ql})}&
(\mu^{-1}(\a)\mu(\b))^{-1}\mu^{-1}(\tqla h_1\tpla)S^{-1}(\tqlb h_2\tplb),
\end{eqnarray*} 
for all $h\in H$. Consequently, 
\begin{eqnarray*}
&&\hspace*{-2cm}
S(h)\rh \l=S(h)\rh \phi\circ S=(\phi\lh S^2(h))\circ S
=(\chi^{-1}S^2(h)\rh \phi)\circ S\\
&=&
(\chi^{-1}S^2(h)\rh \l\circ S^{-1})\circ S
=\l\lh S^{-1}\chi^{-1}S^2(h)\\
&=&
(\mu^{-1}(\a)\mu(\b))^{-1}\mu^{-1}(\tqla h_1\tpla)\l\lh S^{-1}(\tqlb h_2\tplb),
\end{eqnarray*}
finishing our proof.
\end{proof}

{\sl Proof of \equref{normdefmodelem}}. Using the definitions of $q_R$, $q_L$, 
\equref{q3} and \equref{q5} we deduce easily that
\begin{equation}\eqlabel{qlqr}
X^1\ot S(X^2)\tqla X^3_1\ot \tqlb X^3_2=q^1x^1_1\ot S(q^2x^1_2)x^2\ot x^3.
\end{equation}
The Nakayama isomorphism $\xi_{\rm cop}$ for $H^{\rm cop}$ is given by
\[
\xi_{\rm cop}: H^*\ra H~~,~~\xi_{\rm cop}(h^*)=h^*(S^{-1}(\tqla t_1\tpla))\tqlb t_2\tplb~,
\]
and is an isomorphism of left $H$-modules with inverse
$
\xi_{\rm cop}^{-1}(h)=h\rh \Lambda\circ S
$,
where $\Lambda$ is a right cointegral on $H$ satisfying $\Lambda(S(t))=1$, 
see \resref{invxi} (i). 

Now take $h\in H$ and $h^*=\xi_{\rm cop}^{-1}(h)$.
If $q^*:=h^*\circ S^{-1}$ then 
\[
h=q^*(\tqla t_1\tpla)\tqlb t_2\tplb\equal{\equref{qtrversustqlattpla}}
q^*(q^1t_1p^1)q^2t_2p^2S^{-1}(u),
\]
where $u$ is the element introduced in \prref{deforuforint}. Set $p_R=p^1\ot p^2=P^1\ot P^2$ and 
$q_R=q^1\ot q^2=Q^1\ot Q^2$ and compute 
\begin{eqnarray*}
&&\hspace*{-2cm}
\Delta(hS^{-1}(u^{-1}))
=q^*(q^1t_1p^1)q^2_1t_{(2, 1)}p^2_1\ot q^2_2t_{(2, 2)}p^2_2\\
&\equal{\equref{qr2}}&
q^*(q^1Q^1_1x^1t_1p^1)S^{-1}(g^2)q^2Q^1_2x^2
t_{(2, 1)}p^2_1\ot S^{-1}(g^1)Q^2x^3t_{(2, 2)}p^2_2\\
&\equal{(\ref{eq:pr1},\ref{eq:gdi})}&
\mu(X^1)q^*(q^1(Q^1t_1P^1)_1p^1)S^{-1}(g^2)q^2(Q^1t_1P^1)_2p^2S(X^3)f^1\\
&&\hspace*{1cm}
\ot S^{-1}(g^1)Q^2t_2P^2S(X^2)f^2.
\end{eqnarray*} 
This equality is equivalent to 
\begin{eqnarray*}
&&\hspace*{-15mm}
(S^{-1}(f^2)\ot S^{-1}(f^1))\Delta(hS^{-1}(u^{-1}))(g^1\ot g^2)\\
&=&
\mu(X^1)q^*(q^1(Q^1t_1P^1)_1p^1)q^2(Q^1t_1P^1)_2p^2S(X^3)\ot Q^2t_2P^2S(X^2).
\end{eqnarray*}
Applying $\l\ot \Id_H$ to this formula, we find
\begin{eqnarray*}
&&\hspace*{-1.5cm}
\l(S^{-1}(f^2)(hS^{-1}(u^{-1}))_1g^1S(h'))
S^{-1}(f^1)(hS^{-1}(u^{-1}))_2g^2\\
&=&\mu(X^1)q^*(q^1(Q^1t_1P^1)_1p^1)\l(q^2(Q^1t_1P^1)_2p^2S(h'X^3))Q^2t_2P^2S(X^2)\\
&\equal{\equref{lcointsimpl}}&
\mu(x^1X^1)q^*(\tqlb x^3h'_2X^3_2\tplb)\l(S^{-1}(\tqla)Q^1t_1P^1S(x^2h'_1X^3_1\tpla))Q^2t_2P^2S(X^2).
\end{eqnarray*}
We know that $h^*=\xi^{-1}_{\rm cop}(h)=h\rh \Lambda\circ S$, hence 
$q^*=(h\rh \Lambda\circ S)\circ S^{-1}=\Lambda\lh S(h)$.
By \prref{deforuforint} we have that $\Lambda=(\l\circ S^{-1})\lh u^{-1}$, hence $q^*=(\l\circ S^{-1})\lh u^{-1}S(h)$, and this implies that
\begin{eqnarray*}
&&\hspace*{-0.5cm}
\l(S^{-1}(f^2)(hS^{-1}(u^{-1}))_1g^1S(h'))
S^{-1}(f^1)(hS^{-1}(u^{-1}))_2g^2=\mu(x^1X^1)\\
&&\l(S^{-1}(\tqlb x^3h'_2X^3_2\tplb)hS^{-1}(u^{-1}))
\l(S^{-1}(\tqla)Q^1t_1p^1S(x^2h'_1X^3_1\tpla))Q^2t_2p^2S(X^2).
\end{eqnarray*}
Since $u$ is invertible it follows that 
\begin{eqnarray*}
&&\hspace*{-2cm}
\l(S^{-1}(f^2)h_1g^1S(h'))
S^{-1}(f^1)h_2g^2=\mu(x^1X^1)
\l(S^{-1}(\tqlb x^3h'_2X^3_2\tplb)h)\\
&&\hspace*{5mm}
\l(S^{-1}(\tqla)Q^1t_1p^1S(x^2h'_1X^3_1\tpla))Q^2t_2p^2S(X^2)\\
&\equal{(\ref{eq:movingelem1},\ref{eq:ca})}&
\mu(x^1X^1g^1S(x^2_2h'_{(1, 2)}X^3_{(1, 2)}\tilde{p}^1_2)f^1)\l(S^{-1}(\tqlb x^3h'_2X^3_2\tplb)h)\\
&&\hspace*{5mm}\l(S^{-1}(\tqla)Q^1t_1p^1)Q^2t_2p^2S(X^2g^2S(x^2_1h'_{(1, 1)}X^3_{(1, 1)}\tilde{p}^1_1)f^2)\\
&\equalupdown{(\ref{eq:pl1},\ref{eq:f4})}{(\ref{eq:q1},\ref{eq:movingelem1})}&
\mu(\tilde{q}^1_1x^1)\mu(S((\tilde{q}^1_2x^2)_2y^2(h'_2\tplb)_1\tPla)f^1)
\l(S^{-1}(\tqlb x^3y^3(h'_2\tplb)_2\tPlb)h)\\
&&\hspace*{5mm}\l(Q^1t_1p^1)Q^2t_2p^2
S(S((\tilde{q}^1_2x^2)_1y^1h'_1\tpla)f^2)\\
&\equal{(\ref{eq:fu5},\ref{eq:ql2},\ref{eq:qlqr})}&
\l(S^{-1}(\tqlb y^3(x^3_2h'_2\tplb)_2\tPlb)h)\mu(S(G^2_2S(q^1x^1_1)_2\tilde{q}^1_2y^2(x^3_2h'_2\tplb)_1\tPla)f^1)\\
&&\hspace*{5mm} \mu(G^1S(q^2x^1_2)x^2)S^{-1}(\un{g}^{-1})
S(S(G^2_1S(q^1x^1_1)_1\tilde{q}^1_1y^1x^3_1h'_1\tpla)f^2)\\
&\equal{(\ref{eq:ca},\ref{eq:ql2})}&
\mu(G^1S(q^2x^1_2)x^2)
\mu(S(G^2_2g^2S(X^1q^1_1x^1_{(1, 1)})\tqla(\tQlb X^3_2x^3_2h'_2\tplb)_1\tPla)f^1)\\
&&\hspace*{5mm}\l(S^{-1}(\tqlb (\tQlb X^3_2x^3_2h'_2\tplb)_2\tPlb)h)S^{-1}(\un{g}^{-1})\\
&&\hspace*{5mm}S(S(G^2_1g^1S(X^2q^1_2x^1_{(1, 2)})\tQla X^3_1x^3_1h'_1\tpla)f^2)\\
&\equal{(\ref{eq:qlqr},\ref{eq:inchileftcoint},\ref{eq:qr1a})}&
\hspace*{-2mm}
\mu^{-1}(\a)\mu(\b)\mu(G^1S(q^2(y^1_1x^1)_2)y^1_2x^2)
\mu(S(G^2_2g^2S(Q^1q^1_1(y^1_1x^1)_{(1, 1)}))f^1)\\
&&\hspace*{-2mm}
\l(hS(y^3x^3_2h'_2\tplb))S^{-1}(\un{g}^{-1})S(S(G^2_1g^1S(Q^2q^1_2(y^1_1x^1)_{(1, 2)})y^2x^3_1h'_1\tpla)f^2)\\
&\equalupdown{(\ref{eq:qr2},\ref{eq:pf})}{(\ref{eq:q1},\ref{eq:q3})}&
\mu^{-1}(\a)\mu(\b)\mu(Y^1g^1_1G^1S((q^2y^1_2)_2)F^1y^2_1x^1)\mu(S(Y^3g^2S(q^1y^1_1))f^1)\\
&&
\l(hS(y^3x^3h'_2\tplb))S^{-1}(\un{g}^{-1})S(S(Y^2g^1_2G^2S((q^2y^1_2)_1)F^2y^2_2x^2h'_1\tpla)f^2)\\
&\equal{(\ref{eq:ca},\ref{eq:q1},\ref{eq:elemsUandV})}&
\mu^{-1}(\a)\mu(\b)\mu(f^1)\mu(S(y^1)_1Y^1U^1_1y^2_1x^1)\mu^{-1}(S(y^1)_{(2, 2)}Y^3U^2)\\
&&\hspace*{5mm}\l(hS(y^3x^3h'_2\tplb))S^{-1}(\un{g}^{-1})S(S(S(y^1)_{(2, 1)}Y^2U^1_2y^2_2x^2h'_1\tpla)f^2)\\
&\equal{\equref{s4equivversion}}&
\mu(\b f^1)\mu^{-1}(Y^3U^2\a)\mu(Y^1U^1_1y^2_1x^1)
\l(hS(y^3x^3h'_2\tplb))\\
&&\hspace*{5mm}S^{-1}(\un{g}^{-1}y^1)S(S(Y^2U^1_2y^2_2x^2h'_1\tpla)f^2).
\end{eqnarray*}
Thus we have shown that 
\begin{eqnarray}
&&\hspace*{-1cm}
\l(S^{-1}(f^2)h_1g^1S(h'))
S^{-1}(f^1)h_2g^2=\mu(\b f^1)\mu^{-1}(Y^3U^2\a)\nonumber\\
&&\hspace*{5mm}
\mu(Y^1U^1_1y^2_1x^1)\l(hS(y^3x^3h'_2\tplb))
S^{-1}(\un{g}^{-1}y^1)S(S(Y^2U^1_2y^2_2x^2h'_1\tpla)f^2),\eqlabel{fvfformunim}
\end{eqnarray}
for all $h, h'\in H$. Finally, by \cite[Lemma 3.13]{hn3} we have 
\[
x^1U^1\ot x^2U^2_1\mathbf{U}^1\ot x^3U^2_2\mathbf{U}^2=
S(X^1)_{(1, 1)}U^1_1X^2\ot S(X^1)_{(1, 2)}U^1_2X^3\ot S(X^1)_2U^2.
\]
Substituting this formula in \equref{fvfformunim} and applying
(\ref{eq:s4equivversion}, \ref{eq:q3}), we easily obtain \equref{normdefmodelem},
as desired.\cctd

Our proof for the fact that $T\Join r$ is a right integral in $D(H)$ requires the 
following formulas.

\begin{lemma}
Let $H$ be a finite dimensional quasi-Hopf algebra $H$, $h\in H$ and $r\in \int_r^H$.
Then
\begin{eqnarray}
&&\hspace*{3mm}
X^1_1x^1\d^1 S(X^3_2)\ot X^1_2x^2\d^2_1S(X^3_1)_1\ot X^2x^3\d^2_2S(X^3_1)_2\nonumber\\
&&\hspace*{1.5cm}=
(\b S(X^3))_1g^1S(x^3)\ot (\b S(X^3))_2g^2S(x^2)f^1\ot x^2X^1\b S(x^1X^2)f^2~;\eqlabel{app2}\\
&&\hspace*{3mm}
f^2V^1S^{-1}(f^1)_1\ot V^2S^{-1}(f^1)_2=q_L~;\eqlabel{app2a}\\
&&\hspace*{3mm}
S(U^1)\tqla U^2_1\ot \tqlb U^2_2=f~;\eqlabel{app2aa}\\
&&\hspace*{3mm}
S(p^1)F^2f^2_2X^3\ot S(p^2f^1X^1)F^1f^2_1X^2=1\ot \a~,\eqlabel{app2b}\\
&&\hspace*{3mm}
V^1r_1\ot \un{g}^{-1}V^2r_2=V^2r_2p^2\ot S^2(V^1r_1p^1)\a ~,\eqlabel{app4}\\
&&\hspace*{3mm}
S(\tplb)f^1r_1\ot \un{g}^{-1}S(\tpla)f^2r_2=\mu(S(p^2)f^1)S(p^1)f^2V^2r_2P^2
\ot S^2(V^1r_1P^1)\a .\eqlabel{app3b}
\end{eqnarray}
\end{lemma}

\begin{proof}
We have 
\begin{eqnarray*}
&&\hspace*{-2.2cm}
X^1_1x^1\d^1 S(X^3_2)\ot X^1_2x^2\d^2_1S(X^3_1)_1\ot X^2x^3\d^2_2S(X^3_1)_2\\
&\equal{(\ref{eq:gdf},\ref{eq:ca},\ref{eq:q1})}&
(X^1\b_1)_1x^1g^1S(X^3_2)\ot (X^1\b_1)_2x^2g^2_1G^1S(X^3_{(1, 2)})f^1\\
&&\hspace*{5mm} \ot X^2\b_2x^3g^2_2G^2S(X^3_{(1, 1)})f^2\\
&\equal{(\ref{eq:pf},\ref{eq:q1})}&
(X^1\b_1g^1)_1G^1S(x^3X^3_{(2, 2)})\ot (X^1\b _1g^1)_2G^2S(x^2X^3_{(2, 1)})f^1\\
&&\hspace*{5mm}\ot X^2\b_2g^2S(x^2X^3_1)f^2\\
&\equal{(\ref{eq:gdf},\ref{eq:ca})}&
(X^1\d^1S(X^3_2))_1G^1S(x^3)\ot (X^1\d^1S(X^3_2))_2G^2S(x^2)f^1\\
&&\hspace*{5mm}
\ot X^2\d^2 S(x^1X^3_1)f^2\\
&\equal{\equref{delta}}&(\b S(X^3))_1g^1S(x^3)\ot (\b S(X^3))_2g^2S(x^2)f^1\ot x^2X^1\b S(x^1X^2)f^2,
\end{eqnarray*}
and this proves \equref{app2}. 
The equalities \equref{app2a} and \equref{app2b} follow from the definitions 
of $V$ and $p_R$, and from (\ref{eq:pf}, \ref{eq:fgab}).
The verification of all these details 
is left to the reader. We show now \equref{app2aa} by computing
\begin{eqnarray*}
&&\hspace*{-2cm}
S(U^1)\tqla U^2_1\ot \tqlb U^2_2
\equal{\equref{UVpql}}S(\tilde{Q}^1_1p^1)\tqla \tilde{Q}^1_{(2, 1)}p^2_1S(\tQlb)_1\ot \tqlb \tilde{Q}^1_{(2, 2)}p^2_2S(\tQlb)_2\\
&\equal{\equref{ql1a}}&S(p^1)\tqla p^2_1S(\tQlb)_1\ot \tQla \tqlb p^2_2S(\tQlb)_2\\
&=&S(\tplb)f^1S(\tQlb)_1\ot \tQla S(\tpla)f^1S(\tQlb)_2\\
&\equal{\equref{ca}}&S(\tilde{Q}^2_2\tplb)f^1\ot \tQla S(\tilde{Q}^2_1\tpla)f^2\equal{\equref{pqla}}f.
\end{eqnarray*}
In the third equality we have used the second equation in \equref{UVpql}, applied to
$H^{\rm cop}$, namely 
$\tqlb p^2_2\ot S(p^1)\tqla p^2_1=V_{\rm cop}=S(\tpla)f^2\ot S(\tplb)f^1$.

In order to show \equref{app4}, notice that \equref{tsFrobelem} and \equref{firstRadformforquasi} 
imply that
\[
V^1r_1U^1\ot \un{g}^{-1}V^2r_2U^2=V^2r_2U^2\ot S^2(V^1r_1U^1).
\]
As we have already observed, \equref{UVpql} guarantees that  
$\Delta(r)U=\Delta(r)p_R$, and so 
\begin{eqnarray*}
&&\hspace*{-2cm}
V^1r_1\ot \un{g}^{-1}V^2r_2 \equal{\equref{pqr}} V^1r_1q^1_1p^1\ot \un{g}^{-1}V^2r_2q^1_2p^2S(q^2)\\
&=&V^1r_1p^1\ot \un{g}^{-1}V^2r_2p^2\a 
=V^2r_2p^2\ot S^2(V^1r_1p^1)\a ~,
\end{eqnarray*} 
as required. Finally, we have
\begin{eqnarray*}
&&\hspace*{-2cm}
S(\tplb)f^1r_1\ot \un{g}^{-1}S(\tpla)f^2r_2
\equal{\equref{formtplfversusqg}}
\mu^{-1}(g^1)q^1r_1\ot S^{-1}(g^2)q^2r_2\\
&\equal{\equref{qqlv}}&\mu^{-1}(g^1)\mu(\tqla)\tqlb V^1r_1\ot S^{-1}(g^2)V^2r_2\\
&\equal{\equref{rint5}}&\mu(S(g^1)\tqla g^2_1)\tqlb g^2_2V^1r_1\ot \un{g}^{-1}V^2r_2\\
&\equal{(\ref{eq:app4},\ref{eq:fpformula})}&
\mu(S(p^2)f^1)S(p^1)f^2V^2r_2P^2\ot S^2(V^1r_1P^1)\a ,
\end{eqnarray*}
proving \equref{app3b}. This makes the proof complete.
\end{proof}

We are now ready to prove our final result, stating that the Drinfeld double of a 
finite-dimensional quasi-Hopf algebra is unimodular.

\begin{proposition}
If $H$ is a finite dimensional quasi-Hopf algebra, $0\not= r\in \int_r^H$ and 
$0\not=\l \in {\cal L}$ then $\mathbb{T}=\mu^{-1}(\d^2)\d^1\rh \l \Join r$ is a non-zero right integral 
in $D(H)$. Consequently, $D(H)$ is a unimodular quasi-Hopf algebra. 
\end{proposition}

\begin{proof}
By the definitions of $p_L$ and $p_R$, and the axioms (\ref{eq:q3}, \ref{eq:q5})
we obtain that
\begin{equation}\eqlabel{tplvspr}
x^1\ot x^2S(x^3_1\tpla)\ot x^3_2\tplb =X^1_1p^1\ot X^1_2p^2S(X^2)\ot X^3.
\end{equation} 
Now we can prove directly that $\mathbb{T}$ is a right integral in $D(H)$: 
for $\v\in H^*$ and $h\in H$, we compute that
\begin{eqnarray*}
&&\hspace*{-13mm}\mathbb{T}(\v\Join h)
\equalupdown{(\ref{eq:mdd},\ref{eq:ome})}{(\ref{eq:f4},\ref{eq:gdim})}
(y^1X^1_1x^1\d^1S(X^3_2)\rh \l \lh S^{-1}(f^2))(y^2(X^1_2x^2\d^2_1S(X^3_1)_1)_1\\
&&\hspace*{5mm}
r_{(1, 1)}\rh \v\lh S^{-1}(f^1X^2x^3\d^2_2S(X^3_1)_2))\Join y^3(X^1_2x^2\d^2_1S(X^3_1)_1)_2r_{(1, 2)}h\\
&\equalupdown{(\ref{eq:app2},\ref{eq:q1})}{(\ref{eq:ca},\ref{eq:pf})}&
(\b S(X^3))_{(1, 1)}g^1_1G^1S(x^3y^3)\rh \l\lh S^{-1}(\mathbb{F}^2))((\b S(X^3))_{(1, 2)}\\
&&\hspace*{5mm}
g^1_2G^2S(x^2_2y^2)F^1f^1_1r_{(1, 1)}\rh \v \lh S^{-1}(\mathbb{F}^1X^1\b 
S(x^1X^2)f^2r_2))\\
&&\hspace*{5mm}
\Join (\b S(X^3))_2g^2S(x^2_1y^1)F^2f^1_2r_{(1, 2)}h\\
&\equal{(\ref{eq:ca},\ref{eq:gdf})}&
((\d^1S(X^3_2))_1G^1S(x^3y^3)\rh \l\lh S^{-1}(\mathbb{F}^2))
((\d^1S(X^3_2))_2G^2\\
&&\hspace*{5mm}
S(x^2_2y^2)F^1f^1_1r_{(1, 1)}\rh \v \lh S^{-1}(\mathbb{F}^1X^1\b S(x^1X^2)f^2r_2))\\
&&\hspace*{5mm}
\Join \d^2 S(x^2_1y^1X^3_1)F^2f^1_2r_{(1, 2)}h\\
&\equal{\equref{normdefmodelem}}&
\v\left(S^{-1}(\un{g}^{-1}z^1_1t^1X^1\b S(x^1X^2)f^2r_2)S(x^2_2y^2S(U^2_1\mathbf{U}^1z^2t^3_1x^3_1y^3_1\tpla)\mathbb{F}^2)\right.\\
&&\hspace*{5mm}
\left.F^1f^1_1r_{(1, 1)}\right)
\mu(\b\mathbb{F}^1)\mu^{-1}(U^2_2\mathbf{U}^2\a)\mu(U^1z^1_2t^2)
~\d^1S(z^3t^3_2x^3_2y^3_2\tplb X^3_2)\rh \l\\
&&\hspace*{5mm}
\Join \d^2 S(x^2_1y^1X^3_1)F^2f^1_2r_{(1, 2)}h\\
&\equal{(\ref{eq:q3},\ref{eq:ql1})}&
\v\left(S^{-1}(\un{g}^{-1}z^1_1t^1X^1\b S(Y^1y^1_1x^1X^2)f^2r_2)S(Y^3y^2S(U^2_1\mathbf{U}^1z^2t^3_1y^3_1\tpla )\mathbb{F}^2)\right.\\
&&\hspace*{5mm}
\left. F^1f^1_1r_{(1, 1)}\right)\mu(\b\mathbb{F}^1)\mu^{-1}(U^2_2\mathbf{U}^2\a)\mu(U^1z^1_2t^2)~
\d^1S(z^3t^3_2y^3_2\tplb x^3X^3_2)\rh \l \\
&&\hspace*{5mm}
\Join \d^2S(Y^2y^1_2x^2X^3_1)F^2f^1_2r_{(1, 2)}h\\
&\equalupdown{(\ref{eq:tplvspr},\ref{eq:q1})}{(\ref{eq:gdim},\ref{eq:q3})}&
\v\left(S^{-1}(\un{g}^{-1}z^1t^1X^1\b S(Y^1p^1_1x^1X^2)f^2r_2)
S(Y^3p^2S(U^2_1\mathbf{U}^1z^3t^2_2)\mathbb{F}^2)\right.\\
&&\hspace*{5mm}
\left. F^1f^1_1r_{(1, 1)}\right)\mu(\b\mathbb{F}^1)\mu^{-1}(U^2_2\mathbf{U}^2\a)\mu(U^1z^2t^2_1)~
\d^1S(t^3x^3X^3_2)\rh \l\\
&&\hspace*{5mm}
\Join \d^2S(Y^2p^1_2x^2X^3_1)F^2f^1_2r_{(1, 2)}h\\
&\equal{(\ref{eq:gdim},\ref{eq:ca},\ref{eq:q1})}&
\v\left(S^{-1}(\un{g}^{-1}z^1t^1X^1\b S(Y^1(z^2_1t^2_{(1, 1)}p^1)_1x^1X^2)f^2r_2)\right.\\
&&\hspace*{5mm}
\left. 
S(Y^3z^2_2t^2_{(1, 2)}p^2S(U^2_1\mathbf{U}^1z^3t^2_2)\mathbb{F}^2)F^1f^1_1r_{(1, 1)}\right) 
\mu(\b \mathbb{F}^1)\mu^{-1}(U^2_2\mathbf{U}^2\a)\\
&&\hspace*{5mm}
\mu(U^1)~\d^1S(t^3x^3X^3_2)\rh \l\Join \d^2S(Y^2(z^2_1t^2_{(1, 1)}p^1)_2x^2X^3_1)F^2f^1_2r_{(1, 2)}h\\
&\equal{(\ref{eq:qr1},\ref{eq:q3})}&
\v\left(S^{-1}(\un{g}^{-1}z^1X^1\b S(Y^1(z^2_1p^1)_1X^2)f^2r_2)
S(Y^3z^2_2p^2S(U^2_1\mathbf{U}^1z^3)\mathbb{F}^2)\right.\\
&&\hspace*{5mm}
\left. F^1f^1_1r_{(1, 1)}\right)\mu(\b \mathbb{F}^1)\mu^{-1}(U^2_2\mathbf{U}^2\a)\mu(U^1)~
\d^1\rh \l\\
&&\hspace*{5mm}
\Join \d^2S(Y^2(z^2_1p^1)_2X^3)F^2f^1_2r_{(1, 2)}h\\
&\equal{(\ref{eq:q1},\ref{eq:peq},\ref{eq:qr1})}&
\v\left(S^{-1}(\un{g}^{-1}z^1X^1\b S(z^2_1t^1X^2)f^2r_2)
S((z^2_2t^2X^3_1)_2p^2S(U^2_1\mathbf{U}^1z^3t^3X^3_2)\mathbb{F}^2)\right.\\
&&\hspace*{5mm}
\left. F^1f^1_1r_{(1, 1)}\right)\mu(\b \mathbb{F}^1)\mu^{-1}(U^2_2\mathbf{U}^2\a)\mu(U^1)~\d^1\rh \l\\
&&\hspace*{5mm}
\Join \d^2S((z^2_2t^2X^3_1)_1p^1)F^2f^1_2r_{(1, 2)}h\\
&\equal{(\ref{eq:q3},\ref{eq:ql},\ref{eq:ca})}&
\v\left(S^{-1}(\un{g}^{-1}S(\tpla)f^2r_2)S(p^2S(U^2_1\mathbf{U}^1)\mathbb{F}^2)F^1(S(\tplb)f^1r_1)_1\right)
\mu(\b \mathbb{F}^1)\\
&&\hspace*{5mm}
\mu^{-1}(U^2_2\mathbf{U}^2\a)\mu(U^1)~\d^1\rh \l\Join \d^2S(p^1)F^2(S(\tplb)f^1r_1)_2h\\
&\equal{(\ref{eq:app3b},\ref{eq:ca})}&
\v\left(S^{-1}(\a)S(p^2S(\mathbf{U}^1)\mathbb{F}^2S(U^2)_2V^1r_1\mathbb{P}^1)F^1(S(P^1)f^2V^2r_2\mathbb{P}^2)_1\right)\\
&&\hspace*{5mm}
\mu(S(P^2)f^1)\mu(\b\mathbb{F}^1)\mu^{-1}(U^2\a)\mu(S(U^2)_1)~\d^1\rh \l\\
&&\hspace*{5mm}
\Join \d^2S(p^1)F^2(S(P^1)f^2V^2r_2\mathbb{P}^2)_2h\\
&\equalupdown{(\ref{eq:gdim},\ref{eq:rint5})}{(\ref{eq:elemsUandV},\ref{eq:qr})}&
\mu(\b\mathbb{F}^1)\mu^{-1}(\b)\mu(\a)\v\left(S^{-1}(\a)S(p^2S(\mathbf{U}^1)\mathbb{F}^2V^1r_1\mathbb{P}^1)F^1(V^2r_2\mathbb{P}^2)_1\right)\\
&&\hspace*{5mm}
\mu^{-1}(\mathbf{U}^2\a)~\d^1\rh \l\Join \d^2S(p^1)F^2(V^2r_2\mathbb{P}^2)_2h\\
&\equal{(\ref{eq:gdim},\ref{eq:app2a},\ref{eq:mumuinv})}&
\mu^{-1}(\mathbf{U}^2)\v\left(S^{-1}(\a)S(p^2S(\mathbf{U}^1)\tqla r_1\mathbb{P}^1)F^1(\tqlb r_2\mathbb{P}^2)_2\right)\\
&&\hspace*{5mm}
\d^1\rh \l\Join \d^2S(p^1)F^2(\tqlb r_2\mathbb{P}^2)_2h\\
&\equalupdown{(\ref{eq:gdim},\ref{eq:app2aa})}{(\ref{eq:pr1},\ref{eq:q1})}&
\v\left(S^{-1}(\a)S(p^2f^1X^1(r_1p^1)_1P^1)F^1f^2_1X^2(r_1p^1)_2P^2\right)\\
&&\hspace*{5mm}
\d^1\rh \l\Join \d^2S(p^1)F^2f^2_2X^3r_2p^2h\\
&\equal{\equref{app2b}}&
\v\left(S^{-1}(\a)S((r_1p^1)_1P^1)\a (r_1p^1)_2P^2\right)~\d^1\rh \l\Join \d^2r_2p^2h\\
&\equal{(\ref{eq:q5},\ref{eq:q6})}&
\v(S^{-1}(\a))~\d^1\rh \l\Join \d^2r\b h\\
&=&\v(S^{-1}(\a))\va(h)\mu^{-1}(\d^2)\rh \l\Join r
=\va_D(\v\Join h)T\Join r,
\end{eqnarray*}
as required. This finishes the proof.
\end{proof}

%%%%%%%%%%%%%%%%%%%%%%%%%%%%%%%%%%%%%%%%%%%%%%


\begin{thebibliography}{99}
%%%%%%%%%%%%%%%%%%%%%%%%%%%%%%%%%%%%%%%%%%%%
\bibitem{bn}
D. Bulacu, E. Nauwelaerts, {\sl Relative Hopf modules for (dual) quasi-Hopf algebras}, 
J. Algebra {\bf 229}(2) (2000), 632--659.
\bibitem{bc1}
D. Bulacu, S. Caenepeel, {\sl Integrals for (dual) quasi-Hopf
algebras. Applications}, J. Algebra {\bf 266}(2) (2003), 552--583.
\bibitem{bn3}
D. Bulacu, E. Nauwelaerts, {\sl Quasitriangular and ribbon quasi-Hopf algebras}, 
Comm. Algebra {\bf 31}(2) (2003), 657--672.
\bibitem{btfact}
D. Bulacu, B. Torrecillas, {\sl Factorizable quasi-Hopf algebras. 
Applications}, J. Pure Appl. Algebra {\bf 194} (2004), no. 1-2, 
39--84.
\bibitem{btqd}
D. Bulacu, B. Torrecillas, {\sl The representation-theoretic rank of the 
doubles of quasi-quantum groups}, J. Pure Appl. Algebra {\bf 212} (2008), No. 4, 919-940.
\bibitem{dri}
V. G. Drinfeld, {\sl Quasi-Hopf algebras}, 
Leningrad Math. J. {\bf 1} (1990), 1419--1457.
\bibitem{eg3}
P. Etingof, S. Gelaki, {\sl Finite dimensional quasi-Hopf algebras with radical 
of codimension 2},  Math. Res. Lett. {\bf 11} (2004), 685--696. 
\bibitem{enoRadS4}
P. Etingof, D. Nikshych, V. Ostrik, {\sl An analogue of Radford's $S^4$ 
formula for finite tensor categories}, Int. Math. Res. Not. {\bf 54} (2004), 
2915--2933.                    
\bibitem{hn1}
F. Hausser, F. Nill, {\sl Diagonal crossed products by duals of
quasi-quantum groups}, Rev. Math. Phys. {\bf 11} (1999), 553--629.
\bibitem{hn2}
F. Hausser, F. Nill, {\sl Doubles of quasi-quantum groups}, Comm. 
Math. Phys. {\bf 199} (1999), 547--589.
\bibitem{hn3}
F. Hausser, F. Nill, {\sl Integral theory for quasi-Hopf algebras},
preprint math. QA/9904164.
\bibitem{kad}
L. Kadison, {\sl An approach to quasi-Hopf algebras via Frobenius coordinates}, 
J. Algebra {\bf 295} (2006), no. 1, 27--43. 
\bibitem{k}
C. Kassel, ``Quantum Groups", {\sl Graduate Texts in Mathematics}
{\bf 155}, Springer Verlag, Berlin, 1995.
\bibitem{maj}
S. Majid, ``Foundations of quantum group theory", {\sl Cambridge University
Press}, 1995.
\bibitem{pareigis1}
B. Pareigis, {\sl Non-additive ring and Module theory I. General theory 
of monoids}, Publ. Math. Debrecen {\bf 24} (1977), 189--204.
\bibitem{pan}
F. Panaite, {\sl A Maschke-type theorem for quasi-Hopf algebras}, in
``Rings, Hopf algebras, and Brauer groups" (Antwerp/Brussels, 1996), 201--207,
Caenepeel S. and Verschoren A. (eds.), Lecture Notes in Pure and Appl. Math., 
Vol. {\bf 197}, Dekker, New York, 1998.
\bibitem{radf}
D. E. Radford, {\sl Minimal quasitriangular Hopf algebras}, J. Algebra {\bf 157} (1993), 285--315.
\bibitem{schfreeness}
P. Schauenburg, {\sl A quasi-Hopf algebra freeness theorem},  
Proc. Amer. Math. Soc. {\bf 132}(4) (2004), 965--972 (electronic). 
\bibitem{sch}
H. -J. Schneider, Lectures on Hopf algebras, Universidad de
Cordoba Trabajos de Matematica, Serie ``B'', No. 31/95, Cordoba,
Argentina, 1995.
\end{thebibliography}
\end{document}